\documentclass[a4paper,10pt]{iopart}
\usepackage[utf8x]{inputenc}
\usepackage[slovak, english]{babel}
\usepackage{hyperref}
\usepackage{setstack}
\usepackage{iopams}
\usepackage{amsthm}
\usepackage{graphicx}
\usepackage{srcltx}
\usepackage{color}
\usepackage{subfigure}
\usepackage{comment}
\usepackage[square, numbers, sort&compress]{natbib}
\usepackage{algorithm2e}
\usepackage[font=small,labelfont=bf,tableposition=top]{caption}
\DeclareCaptionType[fileext=los,placement={!ht}]{Algorithm}

\newcommand {\half} {\frac 12}
\def\ml{l\kern-0.035cm\char39\kern-0.03cm}

\newcommand{\vct}[1]{\pmb{#1}}

\newcommand {\vphi } {\varphi}

\newcommand {\domain} {\Omega}
\newcommand {\bndry} {\Gamma}
\newcommand {\normal} {\vct n}
\newcommand {\RR } {{\mathbb R}}
\newcommand {\Ldva} {L^2(\domain)}
\newcommand{\BV}{BV(\domain)}

\newcommand {\norm}[2] {\left\|#1\right\|_{#2}}
\newcommand {\norma}[1] {\left\|#1\right\|}
\newcommand {\spacedef}[2] {\{#1\ |\ #2\}}
\newcommand {\dual}[2] {\left<#1,#2\right>}
\newcommand {\scal}[2]{\left(#1,#2\right)}
\def\refe#1{(\ref{#1})}

\newcommand {\td} {\mathcal{T}}
\newcommand {\sldva} {\sigma_{L^2}}
\newcommand {\spc} {\sigma_{PC}}
\newcommand{\reg}{\mathcal{R}}
\def\ruw{\R_{U,W}}

\newcommand {\dint}[1] {\,\mathrm{d}#1}
\newcommand {\dx} {\dint{x}}
\newcommand {\dtheta} {\dint{\theta}}
\newcommand {\dS} {\dint{S}}

\newcommand{\muinv}{{\mu^{-1}}}

\def\smin{\sigma_{\mathrm{min}}}

\newcommand{\uoa}{u_{\alpha}^{\delta}}
\newcommand{\woab}{w_{\alpha,\beta}^{\delta}}
\newcommand{\uoab}{u_{\alpha,\beta}^{\delta}}
\newcommand{\R}{\mathcal{R}}
\newcommand{\D}{\mathcal{D}}
\newcommand{\vd}{v^{\delta}}
\newcommand{\wl}{z}

\def\ud{\mathrm{d}}
\def\tab{\mathcal{T}_{\alpha,\beta}}
\def\Ta{\mathcal{T}_{\alpha}}
\def\F{F}

\newcommand{\im}{\Lambda}
\def\fid{\mathcal{F}} 

\newcommand {\dvr } {\nabla \cdot }
\newcommand {\rot } {\nabla \times }

\newcommand{\ru}{\mathcal{R}_U}
\newcommand{\rw}{\mathcal{R}_W}
\newcommand{\pu}{\mathcal{P}}

\newcommand{\udagl}{u^\dag_{\lambda}}
\newcommand{\wdagl}{w^\dag_{\lambda}}

\newtheorem{problem}{Problem}[section]
\newtheorem{theorem}{Theorem}[section]
\newtheorem{corollary}{Corollary}[section]
\newtheorem{lemma}{Lemma}[section]
\newtheorem{remark}{Remark}[section]

\begin{document}
\title[On a continuation approach in Tikhonov regularization]{On a continuation
approach in Tikhonov regularization and its application in piecewise-constant
parameter identification}
\author{V Melicher and V Vrábeľ}
\address{Research Group for Numerical Analysis and Mathematical
Modelling, Department of Mathematical Analysis, Galglaan 2, 9000 Gent, Belgium}
\eads{\mailto{Valdemar.Melicher@UGent.be}, \mailto{Vladimir.Vrabel@UGent.be}} 

\begin{abstract}
We present a new approach to convexification of the Tikhonov regularization
using a continuation method
strategy. We embed the original minimization problem into a one-parameter family
of minimization problems. Both the penalty term and the minimizer of the
Tikhonov functional become dependent on a continuation parameter. 

In this way we can independently treat two main roles of the regularization
term, which are stabilization of the ill-posed problem and introduction of the a
priori knowledge. For zero continuation
parameter we solve a relaxed regularization problem, which stabilizes the
ill-posed problem in a weaker sense. The problem is recast to the original
minimization by the continuation method and so the a priori knowledge is
enforced.

We apply this approach in the context of topology-to-shape geometry
identification, where it allows to avoid the convergence of
gradient-based methods to a local minima. We present illustrative results for
magnetic induction tomography which is an example of PDE constrained inverse
problem.
\end{abstract}
\date{}
\maketitle


\thispagestyle{empty}


\section{Introduction}\label{sec:introduction}
In this paper we propose and study a \emph{continuation-based approach}
for the Tikhonov regularization of \emph{ill-posed} problems.

We consider ill-posed problems that can be written in the form
of an operator equation
\begin{equation}\label{eq:general_operator_eq}
\F u = v,
\end{equation}
where $\F:\D(\F)\subseteq U \to V$ is a (in general non-linear) forward
operator, mapping between Banach spaces $U$ and $V$. By $v$ we understand
certain
exact measurements projected on $V$. We assume that only noisy data $v^\delta$
are available, such that $\norm{v-v^\delta}{V}\le\delta,$ where $\delta$ is the
level of noise.

Let us introduce a suitable \emph{regularization} $\R: U \to [0,+\infty]$ with
the domain
$\D(\R):=\spacedef{u\in U}{\R(u)\neq+\infty}.$ It is a proper and convex
functional. The general convention is to
consider only those solutions $u$ to ill-posed operator equation
\refe{eq:general_operator_eq}, where $\R(u)$ is sufficiently small. An
element
$u^\dag$ is called an $\R-$minimizing solution (e.g.\cite{Hofmann2007})
if 
\begin{equation}\label{def:r_min_solution}
\R(u^\dag) = \min\spacedef{\R(u)}{\F u = v} < \infty.
\end{equation}

We follow the classical Tikhonov idea \cite{Engl1996,Morozov1984} and consider
minimizers of functional 
\begin{equation}\label{eq:tikhonov_functional}
 \Ta(u) := \norm{\F(u)-\vd}{V}^2 + \alpha \R(u)
\end{equation}
for a suitable regularization parameter $\alpha>0$,
which depends on both noise level and data, i.e. $\alpha(\delta, \vd).$
The first term in \refe{eq:tikhonov_functional} is called the \emph{fidelity
functional (term)}. It ensures that minima of the Tikhonov functional are
approximate solutions of the operator equation \refe{eq:general_operator_eq},
i.e. the problem which we want to solve in the first place. The
regularization term $\R(u)$ stabilizes the ill-posed problem with respect to
the noise and represent a priori assumptions or expectations that we have about
a desired solution. It practically always enforces the membership of $u$ in a
certain
$U.$ As usual, we denote a minimizer of
\refe{eq:tikhonov_functional} as
\begin{equation}
\uoa := \mathop{\rm{argmin}}\limits_{u\in U}\Ta(u).
\end{equation}
It is a well known fact that under certain reasonable assumptions $\uoa$ are
stable approximations of an $\R-$minimizing solution to
\refe{eq:general_operator_eq}, also in a rather general Banach space setting
\cite{Hofmann2007}. The resulting problem of regularization can be
roughly stated as follows:
\begin{problem}\label{prob:1}
Find a suitable $\alpha$ and the corresponding minimizer $\uoa$ of the Tikhonov
functional 
\refe{eq:tikhonov_functional}, such that $\uoa$ approximates $u^\dag$ as close
as possible.
\end{problem}

The main goal of this paper is to construct a sequence
converging to the \emph{global minimizer} $\uoa$. The biggest challenge is how
to avoid the convergence of a numerical minimization method to a local minimum
of \refe{eq:tikhonov_functional}, which is a common problem for standard
\emph{gradient-based minimization methods} (GBMM). 

The possible reasons for the existence of local minima of
\refe{eq:tikhonov_functional} are triadic: the forward operator $\F$ itself, the
noise in the measurements and the penalty term $\R(u).$ The forward operator is
case-specific and the noise is inherent to ill-posed problems. We have however
full freedom of choice of regularization.

When a GBMM is applied to \refe{eq:tikhonov_functional}, the whole resulting
minimizing sequence belongs to $U$. This is enforced by the regularization
$\R(u).$ However the underlying direct problem \refe{eq:general_operator_eq}
generally requires a far less regularity of a solution than it is asked by
$\R(u).$ Even if we expect our final solution to belong to $U$, it is not
necessary to consider only minimizing sequences from $U.$ This restriction is
often the reason that a GBMM converges to a local minimum.

Let us recall that the purpose of adding the regularization is to stabilize
the ill-posed problem and to ensure the desired properties of the solution. The
main idea of the article is to provide these two roles of the regularization
term gradually.


\subsection{Continuation immersion approach}

Let us consider a Banach space $W,$ such that $U$ is a proper subset of $W$
and the problem \refe{eq:general_operator_eq} is well defined in $W$,
i.e. $U\subsetneq
W$ and $\D(F)\cap W \neq \emptyset.$ We can introduce a new Tikhonov functional
analogical to
\refe{eq:tikhonov_functional}
\begin{equation}\label{eq:AuxilFunc}
    \td_\beta(w) := \norma{F(w)-v^\delta}^2 + \beta \reg_W(w),
\end{equation}
with a regularization term $\rw: W \to [0,+\infty]$ and regularization parameter
$\beta>0$.
It is again a convex and proper functional with the domain
$\D(\R_W):=\spacedef{w\in
W}{\R_W(w)\neq+\infty}.$

The main idea is to continuously transform the relaxed functional $\td_\beta$ to
the original $\Ta$ together with the corresponding minimization problems
by making use of the continuation method \cite{Allgower2003}. 
We will stabilize the problem \refe{eq:general_operator_eq} using $W$-based
regularization, i.e. in a ``broader'' sense. Since the Tikhonov regularization
\refe{eq:AuxilFunc} in $W$ is a ``less'' constrained problem  than
\refe{eq:tikhonov_functional}, it will be easier solvable. It will
provide a very good starting point for minimization in $U.$ The
extra desired properties will be progressively imposed on the solution via
continuation-based projection a posteriori. We consider a one-parameter family
of the Tikhonov functionals
\begin{equation}\label{eq:tikhonov_functional_with_continuation}
    \tab(u,w,\lambda) = \norma{F(\wl)-v^\delta}^2 
    + \lambda\alpha \reg_U(u) + (1-\lambda)\beta\reg_W(w),
\end{equation}
where $\lambda \in [0,1]$ and 
\begin{equation}\label{eq:wl}
\wl = \lambda u + (1-\lambda)w.  
\end{equation}
The regularization term $\ru$ stands for the original regularization in
\refe{eq:tikhonov_functional}. The regularizations parameters $\alpha$ and
$\beta$ are in general functions of $\delta$, $\vd$.

The forward problem $\F$ corresponding to
\refe{eq:tikhonov_functional_with_continuation} can 
be understood as acting on the parametrized family $\wl \in W.$ The
regularization part
\begin{equation}\label{eq:reg_func_continuation}
\ruw(u, w, \lambda) := \lambda\alpha\R_U(u) + (1-\lambda)\beta\R_W(w) 
\end{equation}
is better to be understood as a function on $U\times W.$

We consequently deal with a one-parameter family of
minimization problems
\footnote{If $\alpha \equiv \beta$, the above formulas might bring the augmented
Lagrangian method to mind.
Among the differences between these two method, we stress that we minimize here
in the two independent variables $u$ and $v$.
This turns out very convenient, mainly from the numerical point of view, as we will
show later.}.
For $\lambda \in (0, 1)$ we look for a couple from
$U\times W$,
which minimizes the functional \refe{eq:tikhonov_functional_with_continuation}, 
that is
\begin{equation*}
    (\uoab(\lambda), \woab(\lambda)) 
    = \mathop{\rm{argmin}}\limits_{(u,w) \in U\times W} \tab(u, w, \lambda).
\end{equation*}
For $\lambda=1$ we get the original minimization problem of $\Ta$ and for
$\lambda=0$ the the problem reduces to the minimization of \refe{eq:AuxilFunc}.
By abuse of notation we sometimes write that $(u,w)$ is a minimizer of $\tab$
for any $\lambda \in [0,1]$ to denote a minimizing couple $(u,w)\in U\times W$
if $\lambda \in (0,1)$ and also to denote a minimizing element $w\in W$ if
$\lambda=0$ or  $u\in U$ if $\lambda=1$.

Analogically to the notion of the $R$-minimizing solution
\refe{def:r_min_solution}, let us define for each given $\lambda\in[0,1]$ an
\emph{$\ruw$-minimizing solution} as a couple $(\udagl, \wdagl)\in U\times W$,
such that 
\begin{displaymath}\label{eq:reguw}
\ruw(\udagl, \wdagl, \lambda) = \min\{\R_{U, W}(u, w, \lambda) : \F(z) = v\}
< \infty.
\end{displaymath}

The article is organized as follows. In Section \ref{sec:framework} we analyze
the continuation approach in general. In Section \ref{sec:BM} we deal with
piecewise-constant parameter identification problems (PCPIPs), which have been
our motivation to study continuation methods in the context of Tikhonov
regularization. 
We review the relevant state of the art in PCPIPs. Then, we introduce topology-to-shape
continuation method (TSCM). In Section \ref{sec:MIT} we apply the TSCM to
magnetic induction tomography (MIT), which has many applications, e.g. in
biomedical imaging and non-destructive testing of materials. Section
\ref{sec:tscm_implementation} contains the implementation of the TSCM
and several numerical experiments for MIT are presented in \ref{sec:numer_exp}.

\section{Continuation approach for Tikhonov regularization}
\label{sec:framework}

This section deals with theoretical aspects of the continuation approach for
Tikhonov regularization. The functional $\tab$ defined by
\refe{eq:tikhonov_functional_with_continuation}
is always minimized with respect to the variables $(u, w)$ and the variable
$\lambda \in [0,1]$ is taken as a fixed parameter
\begin{equation}\label{minimization_problem}
    \tab(u, w, \lambda)
    \to\min, \quad\lambda u +(1-\lambda)w=z \in \D(F).
\end{equation}

Throughout the section we make the following assumptions:

\begin{enumerate}
    \item[(A1)]
	Let $V$ be a Hilbert space
	and  $W$ be a reflexive Banach space.
	The space $U$ is a closed reflexive proper subspace of $W$, $U
\subsetneq W$.
    \item[(A2)] 
	$F:\D(F)\subseteq W \to V$, where $\D(F)$ is closed and convex, and
$\D:=\D(F)\cap U \neq \emptyset$.
	The map $F$ is {\it strongly continuous}, i.e.
	\begin{equation}
	w_n\rightharpoonup w \quad \mbox{ implies }\quad F(w_n)\to F(w).
	\end{equation}
	It is furthermore a $C^1$-map.
    \item[(A3)]
	$\rw:W\to[0,\infty)$ is a $C^2$-map. 
	It holds that $\rw(0)=0, \rw'(0)=0$ and
	the second derivative $\rw''$ satisfies the condition
	\begin{equation*}
	\dual{\rw''(w)h}{h}_{W^*}\geq C\norm{h}{W}^2 
	\end{equation*}
	for any $w,h \in W$, where $C$ is a positive constant.
    \item[(A4)]
	$\ru:U\to[0,\infty)$ is a $C^2$-map.
	It holds  that $\ru(u)\geq\rw(u)$ for any $u\in U$,
	$\ru(0)=0, \ru'(0)=0$ and
	the second derivative $\ru$ satisfies the condition
	\begin{equation*}
	\dual{\ru''(u)h}{h}_{U^*}\geq C\norm{h}{U}^2 
	\end{equation*}
	for any $u,h \in U$, where $C$ is a positive constant.
\end{enumerate}

Under the assumption (A1) and (A2) the strongly continuous operator $F$ is
moreover {\it completely continuous}, i.e. compact and continuous.
This makes the problem \refe{eq:general_operator_eq} ill-posed
(compare with \cite[Theorem 10.1]{Engl1996}). The assumptions (A3) and (A4)
imply that 
the regularizations $\ru$ and $\rw$ are convex proper functionals.

The assertion below provides a classical result 
about the existence of a minimizer of
\refe{eq:tikhonov_functional_with_continuation} and its characterization.


\begin{lemma}[well-posedness]\label{lm:well-possedness}
Assume (A1)-(A4).
Let $\lambda \in [0,1]$ be arbitrary.
Then there exists a minimizer of $\tab$ 
for any $\alpha\geq 0$, $\beta\geq 0$,
which moreover satisfies the necessary condition
\begin{equation}\label{eq:MinCondition}
    \eqalign{
    D_u \tab(u,w,\lambda) &= 0, \\ 
    D_w \tab(u,w,\lambda) &= 0.
    }
\end{equation}
If $\alpha$ and $\beta$ are large enough,
then a critical point of $\tab$ is a local minimizer,
i.e. the condition \refe{eq:MinCondition} is sufficient for a local
minimum.
\end{lemma}
\begin{proof} 
    The proof is a straightforward application of the variational calculus.
    Let $\lambda \in (0,1)$.
    Since $F$ is strongly continuous,
    the fidelity term is weakly lower semicontinuous.
    So are the regularizations $\rw$ and $\ru$
    by the continuity and convexity argument.
    The functional $\tab$ is their conical sum
    and hence it is weakly lower semicontinuous as well.
    
    Now, Taylor's theorem shows for the regularization $\rw$ that
    \begin{equation*}
	\rw(w)=\rw(0)+\dual{\rw'(0)}{w}_{W^*} 
	+ \int^1_0(1-\theta)\dual{\rw''(\theta w)w}{w}_{W^*}\dtheta
    \end{equation*}
    and so from the assumption (A3) we conclude
    \begin{equation}\label{eq:w_bound}
	\rw(w)\geq C \norm{w}{W}^2
	\qquad\mbox{ for any } w\in W.
    \end{equation}
    Analogously, it follows from the assumption (A4) that
    \begin{equation}\label{eq:u_bound}
	\ru(u)\geq C \norm{u}{U}^2
	\qquad\mbox{ for any } u\in U.
    \end{equation}
    This shows that the functional $\tab$ is also weakly coercive, i.e. 
    \begin{equation*}
	\tab(u,w,\lambda) 
	> C\left (\lambda\alpha\norm{u}{U}^2 
	+ (1-\lambda)\beta\norm{w}{W}^2\right) \to \infty
    \end{equation*}
    as $\norm{u}{U} + \norm{w}{W} \to \infty$. Both properties of $\tab$
together imply that
    the functional $\tab$ attains its minimum (cf. \cite[Theorem
25.D]{Zeidler1989}).
    As $\tab$ is differentiable,
    a minimizer solves the equation \refe{eq:MinCondition}.
    The case when $\lambda =0$ and $\lambda =1$ follows the same lines.
     
    The second derivative of $\tab$ with respect to $u$ and $w$ is positive
    for some sufficiently large $\alpha$ and $\beta$
    which implies 
    that every solution of \refe{eq:MinCondition} is a local minimizer.
\end{proof}

Expanding the condition \refe{eq:MinCondition} for $\lambda \in (0,1)$ 
reveals
\footnote{Note that $F':W\to L(W,V)$, and so $F'(z) \in L(W,V)$ for $z \in W$
and
$F'(z)h \in V$ for $h\in W$}
\begin{equation*}
    \eqalign{
    &\Big[2\scal{F'(\wl) \,\cdot}{F(\wl)-v}
    +\alpha \dual{\reg'_U(u)}{\cdot}_{U^*} \Big]\lambda = 0, \\ 
    &\Big[2\scal{F'(\wl) \,\cdot}{F(\wl)-v}
    +\beta\dual{\reg'_W(w)}{\cdot}_{W^*} \Big](1-\lambda) =0,
    }
\end{equation*}
and thus
\begin{equation*}
    \alpha \dual{\reg'_U(u)}{\cdot}_{U^*} 
    = \beta\dual{\reg'_W(w)}{\cdot}_{W^*}.
\end{equation*}
We use the above formula to establish the so-called Ritz projection 
from the space $W$ to its subspace $U$,
which will turn out useful.
\begin{lemma}[Ritz projection]\label{thm:RitzProjection}
Assume (A1), (A3) and (A4).
Let $u \in U$ be the solution of the problem 
\begin{equation}\label{eq:RitzProblem}
    \alpha\dual{\ru'(u)}{h}_{U^*} = \beta\dual{\rw'(w)}{h}_{W^*} 
    \quad  \mbox{ for all } h \in U,
\end{equation}
where $w \in W$ and $\alpha, \beta>0$ are fixed. 
Then, 
\begin{itemize}
  \item[(i)]  the map $\pu:W\to U$ such that $ w \mapsto \pu (w) = u$ is
well-defined,
  \item[(ii)] the map $\pu$ is continuously differentiable
      with $\pu' =  [\ru''(\pu(w))]^{-1} \circ \frac{\beta}{\alpha} \rw''$,
  \item[(iii)] the a priori estimate $\norm{u}{U}\leq C\norm
{\rw'(w)}{L(W,W^*)}$ holds true.
\end{itemize}
\end{lemma}
\begin{proof}
    (i) It is sufficient to prove the unique solvability of the problem
\refe{eq:RitzProblem}.
    Since $U\subset W$, it follows that $W^* \subset U^*$, and hence $\rw'(w)\in
    U^*$.
    The assumption (A4) implies that the operator $\ru': U\to U^*$
    is hemicontinuous, i.e. $t\mapsto \dual{\ru'(u_1+tu_2)}{h}_{U^*}$
    is continuous on $[0,1]$ for all $u_1, u_2, h \in U$.
    We furthermore deduce that
    \begin{equation*}
	\eqalign{
	&\dual{\ru'(u_1)-\ru'(u_2)}{u_1-u_2}_{U^*} \\
	&=\dual{\int^1_0\ru''(u_1+\theta(u_2-u_1))(u_1-u_2)\dtheta}{u_1-u_2}_{
U^*} \\
	&=\int^1_0\dual{\ru''(u_1+\theta(u_2-u_1))(u_1-u_2)}{u_1-u_2}_{U^*}
\dtheta \\
	&\geq C\norm{u_1-u_2}{U}^2,}
    \end{equation*}
    which shows that $\ru'$ is strongly monotone and a fortiori coercive.
    The theory of monotone operators (see \cite[Theorem 26.A]{Zeidler1989}) then
    guarantees
    that for any $ w \in W$ there exists a unique $u=\pu(w)$ such that
    \begin{equation}\label{eq:RitzProblem2}
	\alpha\ru'(\pu (w)) = \beta\rw'(w),
    \end{equation}
    and that $[\ru'(u)]^{-1}$ is Lipschitz continuous.
    
    (ii) We can now apply the local inverse function theorem 
    \cite[Theorem 4.F]{Zeidler1985},
    because the derivative $\ru''(\pu(w))\in L(U,U^*)$ is bijective 
    on account of (A4) and the linear operator theory.
    It is furthermore a global inverse map,
    because $\ru'$ is proper, 
    i.e. the preimage $\ru'(M)$ of any compact set $M$ is also compact 
    (e.g. \cite[Chapter 4]{Zeidler1985}).
    Consequently, the differentiation of \refe{eq:RitzProblem2} yields
    \begin{equation*}
	\pu'(w) = \left[\ru''(\pu(w))\right]^{-1}\circ
\frac{\beta}{\alpha}\rw''(w),
	\qquad w \in W.
    \end{equation*}
    
    (iii) We put $h=u$ in \refe{eq:RitzProblem} to estimate that
    \begin{equation*}
	\eqalign{
   	C\norm{u}{U}^2 &\leq \alpha\dual{\ru'(u)}{u}_{U^*} 
	= \beta\dual{\rw'(w)}{u}_{W^*} \\
	&\leq \beta\norm{\rw'(w)}{L(W,W^*)}\tilde C\norm{u}{U},}
     \end{equation*}
    which concludes the proof.
    \end{proof}

\begin{remark}
The direct consequence of the above considerations is
that the system \refe{eq:MinCondition} is for $\lambda \in (0,1)$ equivalent 
to the system
\begin{equation*}
\begin{array}{rcl}
    D_w \tab(\pu(w),w,\lambda) &=& 0,\\
    \alpha\ru'(\pu (w)) &=& \beta\rw'(w),
\end{array} 
\end{equation*}
and for $\lambda=0$ we can still define ``the minimizer'' $\uoab(0)$
as the projection $\pu(\woab(0))$.
\end{remark}

The following theorem provides the main result of this section.
It establishes a continuous dependence of the minimizer of $\tab$ on
the parameter $\lambda$. The main idea of the proof lies in realizing that the
problem is a saddle point one. We minimize in $U\times W$ and 
maximize in $\lambda.$ Further, the proof follows the standard lines (compare
with \cite{Engl1996}).
\begin{theorem}[Continuous dependence on $\lambda$]\label{thm:main}
Assume (A1)-(A4). Let $\alpha\geq\beta>0$ and $v^\delta \in V$. 
Assume that there exists a unique global minimizer $(\uoab(\lambda),
\woab(\lambda))$
of \refe{eq:tikhonov_functional_with_continuation} for any $\lambda \in [0,1]$
\footnote{As we have mentioned, if $\lambda=0$ and $\lambda=1$,
we consider just $\woab(0)$ and $\uoab(1)$, respectively.}.
Then the mappings
\begin{equation*}
    \eqalign{
    \woab: [0,1) \to W, &\qquad \lambda \mapsto \woab(\lambda), \\
    \uoab: (0,1] \to U, &\qquad \lambda \mapsto \uoab(\lambda)}
\end{equation*}
are continuous.
\end{theorem}

The theorem has an important corollary, which establishes local correctness of
the continuation extension at $\lambda=1:$
\begin{corollary}\label{cor:local_corectness}
Let the assumptions of Theorem \ref{thm:main} be fulfilled. 
If $\lambda\to 1$, then $\uoab(\lambda) \to \uoa$.
\end{corollary}
\begin{proof}
We begin the proof of Theorem \ref{thm:main} with a few estimates for $\tab$,
which will help us later.
It is evident that
\begin{equation}\label{eq:uw_bound}
    \ruw(u, v, \lambda) \leq \alpha\ru(u)+\beta\rw(w)
\end{equation}
for any $u \in U, w \in W$ and $\lambda \in [0,1]$ . 
Conversely, the assumption (A4) and the convexity of $\rw$ imply
\begin{equation*}
    \eqalign{
    \ruw(u,w, \lambda)
    &\geq \alpha\lambda\rw(u)+\beta(1-\lambda)\rw(w) \\
    &\geq \beta[\lambda\rw(u)+(1-\lambda)\rw(w)]\\
    &\geq \beta\rw(\lambda u+ (1-\lambda)w),}
\end{equation*}
which leads to the estimate
\begin{equation}\label{eq:z_bound}
    \norm{F(z)-v^\delta}{V}^2 + \beta\rw(z)
    \leq \norm{F(z)-v^\delta}{V}^2 + \ruw(u,w, \lambda)
\end{equation} 
for any $(u,w) \in \D\times \D(F)$ and $z=\lambda u+ (1-\lambda)w$.
By the mean value theorem we obtain for the fidelity term
\begin{equation}\label{eq:f_bound}
    \eqalign{
    \norm{F(z)-v^\delta}{V}^2
    &=\norm{F(z) \mp F(w) - v^\delta}{V}^2\\
    &\leq \norm{F(z)-F(w)}{V}^2 + \norm{F(w)-v^\delta}{V}^2\\
    &\leq \norm{F'(\xi)(\lambda u + (1-\lambda)w-w)}{V}^2 +
\norm{F(w)-v^\delta}{V}^2\\
    &\leq \left[\norm{F'}{L(S, V)}\lambda\norm{u-w}{W}\right]^2
    + \norm{F(w)-v^\delta}{V}^2,}
\end{equation}
where the set $S$ is the line segment $u+t(w-u),$ $t\in [0,1]$.

Let now $\lambda_k\to\lambda\in[0,1]$ as $k\to\infty$.
Denote by $(u_k, w_k)$ the corresponding global minimizer $(\uoab(\lambda_k),
\woab(\lambda_k))$
and set $z_k=\lambda_k u_k + (1-\lambda_k)w_k$.
By the definition it holds true  of minimizer that
\begin{equation*}
    \tab(u_k, w_k, \lambda_k) \leq \tab(u, w, \lambda_k)
\end{equation*}
for any $(u,w) \in \D\times \D(F)$. 
We can moreover bound the minimum of $\tab$ uniformly for any $\lambda\in [0,1]$
with the estimates \refe{eq:uw_bound} and \refe{eq:f_bound}
\begin{equation}\label{eq:tab_upbound}
    \eqalign{
    \norm{F(z_k)-v^\delta}{V}^2
+\alpha\lambda_k\ru(u_k)+\beta(1-\lambda_k)\rw(w_k) \\
    \leq \norm{F(z)-v^\delta}{V}^2 + \ruw(u, w, \lambda_k) \\
    \leq \left[\norm{F'}{L(S, V)}\norm{u-w}{W}\right]^2 +
\norm{F(w)-v^\delta}{V}^2 + \alpha\ru(u)+\beta\rw(w), 
    }
\end{equation}
where $(u,w) \in \D\times \D(F)$.
This implies combining with \refe{eq:w_bound} and \refe{eq:u_bound} that
\begin{equation*}
    C(1-\lambda_k)\norm{w_k}{W}^2
    \leq\beta(1-\lambda_k)\rw(w_k) \leq \tilde C,
\end{equation*}
and
\begin{equation*}
    C\lambda_k\norm{u_k}{U}^2
    \leq\alpha\lambda_k\ru(u_k)\leq \tilde C.
\end{equation*}
Therefore, the sequences $\{u_k\}$ and $\{ w_k\}$ are bounded in $W$,
unless $\lambda_k \to 0$ and $\lambda_k \to 1$,
where the estimate \refe{eq:tab_upbound} is inapplicable for $\{u_k\}$ and $\{
w_k\}$, respectively.
If $\lambda_k \to 0$,
we can however use Lemma \ref{thm:RitzProjection}
to find
\begin{equation*}
    \norm{u_k}{U} \leq \rw'(w_k) \leq C.
\end{equation*}
and consequently
\begin{equation*}
    \lambda_k u_k \to 0 
    \quad \mbox{ in } U
    \quad \mbox{ as } \lambda_k \to 0.
\end{equation*}
If $\lambda_k \to 1$, 
it follows from
\begin{equation*}
    C(1-\lambda_k)\norm{w_k}{W}^2 
    = C\norm{\sqrt{1-\lambda_k} w_k}{W}^2 \leq \tilde C
\end{equation*}
that
\begin{equation*}
    (1-\lambda_k) w_k \to 0 
    \quad \mbox{ in } W
    \quad \mbox{ as } \lambda_k \to 1.
\end{equation*}
The estimates \refe{eq:z_bound} and \refe{eq:w_bound}
on the other hand force
\begin{equation}\label{eq:tab_lowbound}
    \eqalign{
    \norm{F(z_k)-v^\delta}{V}^2 +\ruw(u_k, w_k, \lambda_k)
    &\geq \norm{F(z_k)-v^\delta}{V}^2 + \beta\rw(z_k)\\
    &\geq \beta C \norm{z_k}{W}^2,
    }
\end{equation}
which together with \refe{eq:tab_upbound} ensures
that the sequence $\{z_k\}$ is always uniformly bounded  in $W$
\begin{equation*}
  \norma{z_k}_W \leq C.
\end{equation*}

Bounded sequences in reflexive spaces are weakly compact 
and so we can choose weakly convergent subsequences
\begin{equation}
    u_m \rightharpoonup \overline {u},\,
    w_m \rightharpoonup \overline{w}\,\mbox{ and }
    z_m \rightharpoonup \overline{z}
    \quad\mbox{ as } m\to\infty.
\end{equation}
The above estimates moreover establish that
\begin{equation*}
    \overline{z} 
    = \lambda\overline{u} + (1-\lambda)\overline{w}
    \quad \mbox { for any } \lambda \in [0,1].
\end{equation*}
We then consecutively deduce by the weak lower semicontinuity of $\tab$ and
the definition of minimizer that
\begin{equation*}
    \eqalign{
    \|F(&\overline {z})-v^\delta\|_V^2 +\ruw(\overline {u},\overline
{w},\lambda)\\
    &\leq \underset{m\to \infty}{\lim\inf}
    \left[\norm{F(z_m)-v^\delta}{V}^2 + \ruw(u_m, w_m, \lambda_m)\right]\\
    &\leq \underset{m\to \infty}{\lim\sup}
    \left[\norm{F(\lambda_m u_m+(1-\lambda_m)w_m)-v^\delta}{V}^2 + \ruw(u_m,
w_m, \lambda_m)\right]\\
    &\leq \lim_{m\to\infty}
    \left[\norm{F(\lambda_m u+(1-\lambda_m)w)-v^\delta}{V}^2 + \ruw(u, w,
\lambda_m)\right]\\
    &= \norm{F(z)-v^\delta}{V}^2 + \ruw(u, w, \lambda)}
\end{equation*}
for all $(u,w) \in \D\times \D(F)$.
This shows
that $(\overline u, \overline w)$ is 
minimizer of \refe{minimization_problem}
and that
\begin{equation}\label{eq:tab_limit}
    \lim_{m\to\infty} \tab(u_m, w_m, \lambda_m)
    =\tab(\overline u, \overline w, \lambda).
\end{equation}
Assume now that $(u_m, w_m) \not\to(\overline u, \overline w)$.
Then $c:=\lim\sup \ruw(u_m,w_m, \lambda)>\ruw(\overline u,\overline w, \lambda)$
and there exists a subsequence $\{(u_n,w_n)\}$ of $\{(u_m,w_m)\}$
such that $(u_n,w_n) \rightharpoonup (\overline u, \overline w)$,
$F(z_n)\rightharpoonup F(\overline z)$ and $\ruw(u_n,w_n, \lambda) \to c$.
As a consequence of \refe{eq:tab_limit}, we obtain
\begin{equation*}
    \eqalign{
    \lim_{n\to\infty} \norm{F(z_n)-v^\delta}{V}
    &= \norm{F(\overline z)-v^\delta}{V} + \ruw(\overline u,\overline w,
\lambda)-c\\
   &<\norm{F(\overline z)-v^\delta}{V},}
\end{equation*}
which is in contradiction with weak lower semicontinuity of the norm.

Since the minimizer $(\overline u, \overline w)$ is unique for any $\lambda \in
[0,1]$, 
the above considerations demonstrate that
every sequence $\{(u_k, w_k)\}$ contains a subsequence 
strongly converging towards $(\overline u, \overline w)$,
and therefore, the functions $\uoab$ and $\woab$ are continuous
on the intervals $(0,1]$ and $[0,1)$, respectively.
\end{proof}
The next two theorems address the questions of stability and convergence
of minimizers of $\tab$.
We omit their proofs, 
because they  go along the same lines 
as e.g. in \cite[Theorem 10.2 and 10.3]{Engl1996}.

\begin{theorem}[stability]\label{th:stability}
Assume (A1)-(A4), $\alpha > 0, \beta>0$ and $v^\delta \in V$. 
Let $\lambda \in [0,1]$ be fixed and let $\{v_k\}$ and $\{(u_k, w_k)\}$ be
sequences
such that $v_k\to v^\delta$ and $(u_k, w_k)$ is a minimizer 
of \refe{eq:tikhonov_functional_with_continuation}
with $v^\delta$ replaced by $v_k$.
Then there exists a convergent subsequence of $\{(u_k, w_k)\}$
and the limit of every convergent subsequence is a minimizer 
of \refe{eq:tikhonov_functional_with_continuation}.
\end{theorem}

\begin{theorem}[convergence]\label{th:convergence}
Assume (A1)-(A4). Let $v^\delta \in V$ with $\norm{v-v^\delta}{V}\leq \delta$
and let $\lambda \in [0,1]$ be fixed.
Let $\alpha(\delta)$ and $\beta(\delta)$ be such that
$\alpha(\delta)\to 0,$ $\beta(\delta) \to 0$ 
and $\delta^2/\alpha(\delta)\to 0,$ $\delta^2/\beta(\delta) \to 0$
as $\delta \to 0$.
Then every sequence $\{(u_{\alpha_k}^{\delta_k}, w_{\beta_k}^{\delta_k})\}$, 
where $\delta_k \to 0,$ $\alpha_k=\alpha(\delta_k),$ $\beta_k=\beta(\delta)$
and $(u_{\alpha_k}^{\delta_k}, w_{\beta_k}^{\delta_k})$ is the solution of
\refe{minimization_problem},
has a convergent subsequence.
The limit of every convergent subsequence is an $\ruw$-minimizing solution.
If in addition, the $\ruw$-minimizing solution $(\udagl, \wdagl)$
is unique, then
\begin{equation*}
    \lim_{\delta\to 0} (u_{\alpha_k}^{\delta_k}, w_{\beta_k}^{\delta_k})
    = (\udagl, \wdagl).
\end{equation*}
\end{theorem}

The last result about 
the existence of an $\ruw$-minimizing solution
is essentially due to \cite{Hofmann2007}.
\begin{lemma}
Assume (A1)-(A4).
If there exists a solution of \refe{eq:general_operator_eq},
then there exists an $\ruw$-minimizing solution for any $\lambda \in [0,1]$.
\end{lemma}
\begin{proof}
    Let $v^\delta=v$ in \refe{eq:tikhonov_functional_with_continuation}
    and consider  the case when $\lambda \in (0,1)$.
    Suppose for the sake of contradiction that
    there does not exist an $\ruw$-minimizing solution in $\D\times \D(F)$. 
    Then there exists a sequence $\{(u_k, w_k)\}$
    of solutions of \refe{eq:general_operator_eq} in $\D\times\D(F)$
    such that
    $\ruw(u_k, w_k, \lambda)\to c$ and
    \begin{equation}\label{proof_ruv_min}
    \eqalign{
	&c < \ruw(u, w, \lambda)\quad \\
	\mbox{ for all } (u,w) &\in U \times V 
	\quad\mbox{satisfying}\quad
	F(\lambda u+(1-\lambda) w)=v.}
    \end{equation}
    For a sufficiently large $k$, it follows that
    $\tab(u_k,w_k, \lambda)=\ruw(u_k, w_k, \lambda)<2c$, 
    and so we see by \refe{eq:u_bound} and \refe{eq:w_bound} that
    \begin{equation}
    C\left(\lambda\alpha\norma{u_k}^2_U + (1-\lambda)\beta\norma{w_k}^2_W\right)
\leq 2c. 
    \end{equation}
    One can thus extract a weakly convergent subsequence, 
    again denoted by $\{(u_k,w_k)\}$, with the limit $(\overline u, \overline
w)$.
    The weak lower semicontinuity of $\ruw$ implies that 
    $\ruw(\overline u, \overline w) \leq \lim \inf_{k\to \infty} \ruw(u_k, w_k,
\lambda) =c$. 
    
    However, the map $F$ is strongly continuous and hence the equality
    $F(\lambda u_k+(1-\lambda) w_k)=v$ forces $F(\lambda \overline u
+(1-\lambda)\overline w)=v$,
    which is the contradiction to \refe{proof_ruv_min}.
    
    The case when $\lambda =0$ and $\lambda =1$ goes along the same lines. One
has to consider only
    $\rw$ and $\ru$ functionals with corresponding $\rw$-minimizing solution
and 
    $\ru$-minimizing solution, respectively.
\end{proof}


\section{Piecewise-constant parameter identification problems}\label{sec:BM}
Our motivation to study minimizers of
\refe{eq:tikhonov_functional_with_continuation} comes from piecewise-constant
parameter identification problems (PIPs). We analyze partial differential
equation (PDE) constrained problems with the unknown parameter being a
coefficient of the PDE-constraint. 

For illustration purposes we consider merely
a double-valued piecewise-constant parameter
\begin{equation}\label{eq:parameter_PC}
    \spc = \sigma_1\chi_D +\sigma_2 \chi_{\domain/D},\quad \sigma_1,\sigma_2 \in
\RR,
\end{equation}
where the domain $\Omega$ is an open bounded set, on which the PDE-constrained
problem is defined. The symbols $\chi_D$ and $\chi_{\domain/D}$  stand for the
characteristic function of subset $D\subset \domain$ and its complement,
respectively.  The goal is to find the subdomain $D$ and the unknown numbers
$\sigma_1$ and $\sigma_2$ based on suitable  observations of the state variable
of the PDE-constraint. A classical example here is the problem of inverse
electric impedance tomography (EIT).

We are primary concerned by building an robust and efficient numerical
algorithm to recover the unknown $\spc$. In the case of EIT, the problematic is
extensively studied in the literature, see a comprehensive review
\cite{Borcea2002}.

Why do we look for the solution in the space of
piecewise constant functions? Such a choice is natural, given a problem like
EIT. First, this class of functions is rich enough in order to be applicable.
Second, as in the case of EIT, one usually has only a finite number of
measurements on the boundary $\bndry$ corresponding to the Neumann-to-Dirichlet
operator. For a two-dimensional domain $\domain$, these measurements are
one-dimensional. It is reasonable to assume, that we can successfully recover at
most a one-dimensional unknown inside the domain. \footnote{We do not claim that
certain two-dimensional recovery is impossible.} This is precisely, what one
does by considering \refe{eq:parameter_PC}. The goal is as a matter of fact to
find the interface between the two regions of $\domain$. It is the choice of
space plays a role of regularization.

$\mathbf{U=\BV}$: The most suitable type of regularization for
piecewise-constant parameter identification problems is the $\BV-$
regularization \cite{Acar1994}.
The space $\BV$ is the subspace of functions $u \in L^1(\Omega)$ such
that the quantity 
\begin{displaymath}
J(u,\Omega)=\sup\left\{\int_\Omega u(x) \dvr\xi(x) \ud x\  |\  \xi\in
C^\infty_c(\Omega, \mathbb R^n), \|\xi\|_{L^\infty(\Omega, \mathbb R^n)}\leq
1\right\},
\end{displaymath}
is finite, where $
C^\infty_c(\Omega, \mathbb R^n)$ is the set of smooth functions in 
$C^\infty(\mathbb
R^n)$ with compact support in $\Omega$.
Endowed with the norm 
\begin{equation}
\norm{u}{BV(\Omega)}:=\norm{u}{L^1(\Omega)} + J(u,\Omega),
\end{equation}
it is a
Banach space. 

Tikhonov regularization formulation for the piecewise-constant PIP 
then reads as
\begin{equation}\label{eq:tikhonov_functional_tts}
    \Ta(\spc) := \norm{\F(\spc)-\vd}{V}^2 + \alpha \norm{\spc}{BV(\Omega)}^2,
\end{equation}
where  $F$ is the operator associated with the forward problem. This functional
is a particular case of the functional \refe{eq:tikhonov_functional} from the
introduction when we set $U=\BV.$

\subsection{State of the art of geometry (shape)
identification}\label{sec:sota_of_geom_identification}

In case the constants $\sigma_1$ and $\sigma_2$ in \refe{eq:parameter_PC} are
identified, 
the piecewise-constant parameter $\sigma$ estimation is equivalent to the
geometry identification of the subdomain $D.$

The classical methods to identify the structural information are mostly
based on a study of the sensitivity of certain cost functional to a
infinitesimal change of the
shape of the structure itself, see \cite{Pironneau1984} and the references
therein. This shape sensitivity approach yields eventually to the notion of
\emph{shape derivative} \cite{Sokolowski1992}.

The methods based on the shape sensitivity approach, level set method
parameterizations including \cite{Santosa1996, Fang2003}, are updating the shape
of domain first, not the topology. The topology is prescribed a priori by an
initial guess. The choice of a good initial guess becomes very important for the
method to converge to the optimal shape. Even if some proposed (and well
designed) algorithms are able to find the optimal shape \cite{Chan2005}, the
convergence
is usually very slow. The speed of the convergence is again strongly dependent
on a good initial guess.

The second class of methods is based on the \emph{homogenization} theory, see
the pioneering work \cite{Bendsoe1988} or the monograph \cite{Allaire2002}. The
optimal geometry is obtained in an enriched space of composite designs. The
corresponding classical design can be retrieved via thresholding or
penalization. This approach overcomes some restrictions of the
classical shape sensitivity approach. Both the topology and shape are optimized
at once. The final acquired geometries are close to the optimal onces.
Unfortunately, this approach is limited to certain types of
problems and its rigorous application is a non-trivial task.

A method based on an iterative inclusion of new holes (so called
``bubbles'') into the geometry was investigated in \cite{Eschenauer1994}. This
idea is actually closely related to the one of the homogenization approach.
In \cite{Schumacher1996}, a pointwise limit of such inclusions was used in
linear elasticity to find a optimal design characterized by the so-called
compliance functional. The importance of this contribution was recognized in
\cite{Sokolowski1997, Sokolowski1999, Sokolowski1999a}, where the idea was
extended to shape functionals and the notion of \emph{topological derivative}
was introduced and further developed. Since the introduction of the topological
derivative, a great number of
contribution were made using this concept both in science and in engineering. We
are interested particularly in those where topological and shape sensitivity
concepts are used in
conjunction.

In \cite{Burger2004a} the authors first considered the shape derivative
based level set method (LSM). The motion of the interface described by the LSM
is governed by a non-linear Hamilton-Jacobi equation, where speed is dependent
on shape derivative of the cost functional, as usual. The idea was to introduce
a new source term into the Hamilton-Jacobi equation, dependent on the
topological derivative. This term allows for nucleation of new holes in the
domain. The approach was generalized in \cite{He2007}.

In \cite{Allaire2004} the authors study shape derivative based level set method
for structural optimization. They do not use the topological derivative in the
work itself, but, to our best knowledge, for the first time the topological
derivative is suggested to be used for initialization of the algorithms based on
the shape sensitivity approach. They study the idea in \cite{Allaire2005}, where
an alternating algorithm using both the shape and the topological derivatives is
proposed.

In \cite{Nielsen2007} the authors propose a variant of a binary level set
approach for solving elliptic problems with piecewise constant coefficients. The
inverse problem is solved by a variational augmented Lagrangian approach with a
total variation regularization. Their implementation was able to recover rather
complicated geometries without assuming anything about $D$ a priori, i.e.
without any initial guess. As we will understand later on, it is due to the
nature of the augmented Lagrangian approach which imposes the piecewise constant
constraint gradually. The results of \cite{Nielsen2007} are applied to
piecewise constant level set method (PCLSM) parametrization in \cite{Zhu2011}.
They are employed to study an optimization problem. The PCLSM methods for the
identification of discontinuous parameters in ill-posed problems are considered
in \cite{Cezaro2013}. Both a Tikhonov regularization approach using operator
splitting techniques and an augmented Lagrangian approach are introduced and
analyzed.

In \cite{Hintermuller2008} topological sensitivity based initial guess is used
as starting point for shape-sensitivity level set method to solve an electric
impedance tomography problem.

\subsection{Topology-to-shape continuation method}
\label{sec:tscm}
In this section we introduce a continuation approach to shape identification
which combines topology and shape sensitivities. 

The main idea is based on the following reasoning. Roughly speaking topological
properties of a particular shape are those which stay invariant under various 
\emph{continuous transformations}\footnote{ In our case, the ``shape'' of the
piecewise constant $\sigma$ defined by
\refe{eq:parameter_PC}, the topology is determined by the number of connected
components of $D$ and their equivalent classes (ball, torus etc.).}. A shape
itself is a certain topology modified by those continuous boundary-like
transformations, see the above section. 
Therefore, the topology is the ``coarse'' information about a particular shape.
In this line of reasoning, it is intuitive to first look for the topology itself
and to consider \emph{continuation} methods to transform it to the particular
shape.

We will consider the relaxed parametrization of $\spc$
\begin{equation}\label{eq:sigma_tscm}
    \sigma = (1-\lambda)\sldva + \lambda\spc 
\end{equation}
analogously to \refe{eq:wl}.
We assume that $\sldva \in L^2(\domain)$, because
the space $U=BV(\Omega)$ is included at most in $W=L^2(\Omega)$, in the case if
the domain $\Omega \subset \mathbb R^2$.

The function $\sldva$ can be interpreted as topological derivative. It is
almost everywhere locally defined and represents the distribution of the mass in
$\domain.$ The optimization with respect to $\sldva$ means adding and removing
mass locally
at a given point in the domain. On the other hand,
the optimization with respect to $\spc$ is driven by shape derivative flux and
moves only the interface $\partial D$.

The regularization functional \refe{eq:reg_func_continuation} becomes
\begin{equation}\label{eq:ruw_tscm}
    \ruw(\spc, \sigma_{L_2}, \lambda) =
(1-\lambda)\beta\norm{\sldva}{L^2(\Omega)}^2 +
\lambda\alpha\norm{\spc}{BV(\Omega)}^2.
\end{equation}
The $\R_W = \norm{\cdot}{L^2(\Omega)}^2$ trivially fulfills the assumption (A4).
The assumption (A2) is dependent on the specific forward problem. For magnetic
induction tomography 
it will be established in Section \ref{sec:MIT}. The problematic assumptions are
(A1) and (A3). First, the space $\BV$ is not reflexive. A direct remedy is to
approximate $\BV$ by 
its reflexive subspace $W^{1+\eta}(\domain)$, $0<\eta\ll 1,$ which resolves also
the non-differentiability of $BV-$norm. The second possibility is to follow the
analysis in \cite{Acar1994}. 
There, the convergence in $\BV$ is understood in weaker then norm topology,
namely in $L^p-$sense \footnote{Interestingly, it is the topology of $W$.}. The
seminorm $J(\sigma)$ in $\BV$ is 
furthermore efficiently approximated by the functional (\cite[Theorem
2.2]{Acar1994})
\begin{equation}\label{eq:bv_functional}
  J_\varepsilon(\sigma) = \int_\domain \sqrt{|\nabla \sigma|^2 +
\varepsilon}\dx, \qquad
\varepsilon>0,
\end{equation}
which is differentiable everywhere. We note that $\varepsilon$ will be used
subsequently in different situations and it always represents a small positive
number.

We conclude that for the admissible forward operator $F$ the topology-to-shape
continuation method lies within the proposed continuation framework (Section
\ref{sec:introduction} and \ref{sec:framework}).

\subsubsection{Contributions of TSCM}

Despite all the effort in combining topology and shape sensitivity concepts and
some very positive results as stated in Section
\ref{sec:sota_of_geom_identification}, no clear idea has yet been presented how
these concepts could be unified in one framework. We quote
\cite{Fulmanski2008}:``It is still an open problem to devise how the
combination of boundary variations and singular perturbations of geometrical
domains enters in a general framework of shape optimization.'' We think that the
idea of continuation extension of Tikhonov regularization presented in this
article provides a framework that connects both concepts. We first identify the
optimal distribution of the unknown parameter which represents the topology. We
then continuously recast this information to the optimal shape. We use no
singular perturbations of the geometry. As a consequence, the difficulties in
coupling the local and global sensitivity concepts vanish. We remark that the
approach of singular perturbations of the geometry \cite{Allaire2005} is more
general. It allows to adapt the topology explicitly during the algorithm's
execution.

The numerical experiments in Section \ref{sec:numer_exp} show that the method
is, at least in certain settings, a globally convergent one. However, we have
been able to proof only a local convergence of TSCM, not the global one.

Let us quote also from \cite{Tai2007}, where a penalty method is used to solve
piecewise constant parameter identification problems:``From our numerical
experiences, we find that it is better to neglect the regularization term at the
beginning stage of the iteration. At this stage, we should let the
output-least-squares term to drag $\phi$\footnote{piecewise constant level set
function} into the right direction without thinking about the regularity of
$q$\footnote{coefficient to be recovered}.'' In the context of continuation it
is easy to explain this observation from \cite{Tai2007}. The minimization
without total variation regularization term essentially behaves as Landweber
type of regularization method, where the number of iterations plays the role of
regularization \cite{Engl1996}, and the method converges to the least square
solution in $L^2-$sense. Gradually increasing regularization parameter in the
front of the total variation term functions as the continuation parameter
$\lambda.$ The same insight explains the global convergence of augmented
Lagrangian methods \cite{Cezaro2013}. The advantage of the continuation approach
is that the relaxed space $W$ does not have to be $L^2(\domain).$


\section{Magnetic induction tomography}\label{sec:MIT}

In this section we apply the framework to an inverse problem in magnetic
induction tomography
(MIT).

MIT is a non-invasive visualization technique, which is a very
promising member of the broader electromagnetic imaging
family. It has many potential
applications, for instance non-destructive testing, industrial and medical
imaging \cite{Griffiths2001}.
We refer the reader to the paper \cite{Soleimani2008} for a comprehensive
review.
Magnetic induction tomography is a non-contact technique,
in contrast to widely studied electrical impedance tomography 
\cite{Cheney1999}. Another advantage of MIT is its explicit frequency
dependence,
which allows for more accurate reconstruction of the body properties
\cite{Brunner2006}.

\subsection{Mathematical formulation}

We proceed to the mathematical description of MIT. Electromagnetic phenomena in
general are governed by the Maxwell equations.
Considering the linear isotropic case, the time-harmonic regime with
the angular velocity $\omega>0$ and making use of the magnetic vector potential
$\vct A$ ($\vct{B}=\rot\vct{A}$), we can write them in the form
\begin{equation}\label{eq:potential}
\begin{array}{rcl}
  \rot(\muinv\rot\vct{A}) + i\omega(\sigma + i\omega\epsilon)\vct{A} &=&
\vct{J}_e, \\
  \dvr(\epsilon\vct{A}) &=& 0.
\end{array}
\end{equation}
The scalar potential $V$ is eliminated by the temporal gauge.
The permeability $\mu$ and the permittivity $\epsilon$ are known
strictly
positive scalar functions of the space variable.
The conductivity $\sigma$ is assumed to be positive in the imaged body and it
vanishes in the surrounding non-conducting region;
$\vct{J}_e$ stands the applied current from the excitation coil.
For more on various MIT models we refer to \cite{Zolgharni2009, Soleimani2008}.

We formulate a simplified MIT boundary value problem.
Let $\domain$ be a bounded two-dimensional domain  in the $xy$-plane
with the sufficiently smooth boundary $\partial\domain=:\bndry$.
It represents a cross section of the imaged body.
Assume that the applied current $\vct{J}_e$ is perpendicular 
to $xy$-plane and does not depend on $z$-coordinate.
The induced eddy currents can be then described by the $z$-component of the
potential
$\vct{A}$ which we will simply denote by $A$. We restrict ourselves to the
imaged body region, where the conductivity is strictly positive,
$\sigma\geq \sigma_{\min} >0$. The domain source $\vct{J}_e$ is modeled by a
boundary source $e$, 
which is imposed via the Neumann boundary condition on
$\bndry$. The corresponding experimental setup is depicted in Figure
\ref{fig:MIT}. For an experimental realization see \cite{Korjenevsky2000}.

\begin{figure}[t]
  \includegraphics[width=0.9\columnwidth]{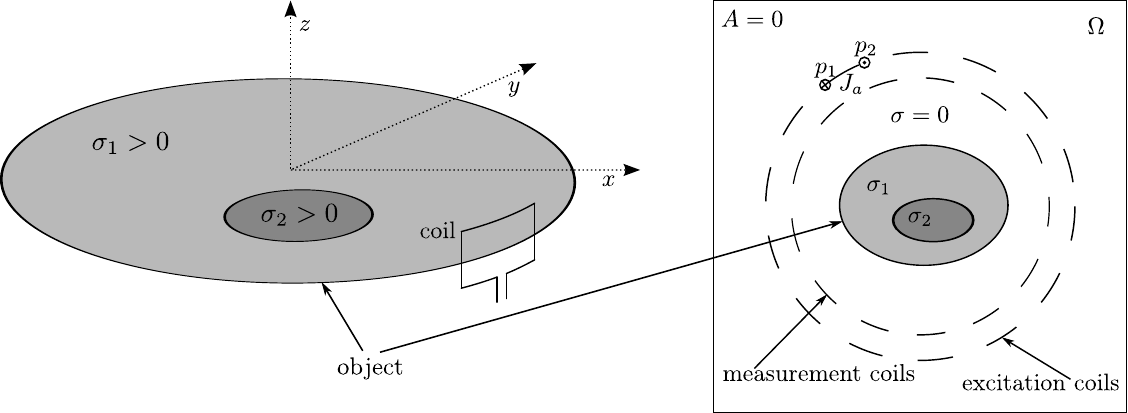}
  \caption{Magnetic induction tomography setup}
  \label{fig:MIT}
\end{figure}

We use the eddy current approximation of the Maxwell equations, where the
displacement current term $i\omega\epsilon \vct{A}$ in \refe{eq:potential} is
disregarded. The state variable $A$ then satisfies the forward problem
\begin{equation}\label{eq:MIT}
\begin{array}{rcl}
  \nabla\cdot(\mu^{-1}\nabla A) + i\omega\sigma A &=& 0 \qquad \mbox{ in }
\domain, \\ \nonumber
  \muinv\nabla A \cdot \normal &=& e  \qquad \mbox{ on } \bndry.
\end{array}
\end{equation}
Let us remark, that under physiological conditions for higher
excitation frequencies $\omega$ the
displacement current term can have a significant contribution
and has to be taken into account.

\subsection{Forward problem}\label{sec:mitforward}

We now show that the MIT forward problem satisfies the assumption (A2) of
Section \ref{sec:framework}. 

Let us first introduce some notation.
The standard scalar product of two complex valued functions in the space
$L^2(\domain)$
is denoted by $(u, v) = \int_\domain u(x) \overline{v(x)}\dx$.
We write $\|u\| = \sqrt{(u,u)}$ for the induced norm.
The subscript $\bndry$ indicates integration over the boundary in
$L^2(\bndry)$-sense.
The symbol $H^1(\domain)$ stands for the Sobolev space
of the complex-valued functions with first weak derivatives.
It is compactly embedded in the all Lebesgue spaces but $L^\infty(\domain)$
(e.g. \cite[Theorem 5.8.2]{Kufner1977}):
\begin{equation}\label{SobolevEmbedding}
    H^1(\domain) \hookrightarrow\hookrightarrow L^q(\domain)
    \quad\mbox{ for any } q\in[1,\infty). 
\end{equation}

The weak formulation of \refe{eq:MIT} reads as
\begin{equation}\label{eq:MITweak}
  \scal{\muinv\nabla A}{\nabla \vphi} + \scal{i\omega\sigma A}{\vphi} 
  = \scal{e}{\vphi}_\bndry 
  \qquad \forall \vphi \in H^1(\domain).
\end{equation}
This variational problem defines
the \emph{impedance map} $\im$, the so-called Neumann-to-Dirichlet map
\begin{equation}\label{eq:NtDmap}
  \im: (\sigma, \omega, e) \mapsto A|_\bndry.
\end{equation}

\begin{lemma}\label{lem:impedance1}
The impedance map
\begin{equation*}
  \im:\sigma \mapsto \im(\sigma)=A|_\bndry,
\end{equation*}
where the function $A$ is the solution of the problem \refe{eq:MITweak} for any
$e \in L^2(\bndry)$ and $\omega>0$ fixed,
is a well-defined and strongly continuous map from the set
\begin{equation*}
    M=\left\{ \sigma \in L^q(\domain), q >1:
    \sigma \geq \sigma_{\min}>0\right\}.
\end{equation*}
to the space $L^2(\bndry)$.
\end{lemma}
\begin{proof}
    
    The  Sobolev embedding \refe{SobolevEmbedding}
    implies that term in \refe{eq:MITweak} containing $\sigma$ makes sense for
    any $\sigma \in L^q(\domain),$ $q>1$.
    Given arbitrary $\sigma \in M$, the existence of a unique solution $A \in
H^1(\domain)$ 
    follows readily from the Lax-Milgram theorem for sesquilinear forms.  
      
    Let now $\sigma_n \rightharpoonup \sigma$ as $n\to\infty$.
    It holds that $\sigma\in M$, because $M$ is closed and convex.
    Denote by $A_n$ and $A$ the corresponding solutions 
    of \refe{eq:MITweak} 
    for $\sigma_n$ and the weak limit $\sigma$ respectively.
    The subtraction of the variational formulas from each other gives
    \begin{equation*}
	\scal{\muinv\nabla(A-A_n)}{\nabla \vphi} +
	\scal{i\omega\sigma(A-A_n)}{\vphi}
	=\scal{i\omega(\sigma_n-\sigma)A}{\vphi}.
    \end{equation*}
    The sesquilinear form on the left hand side is equivalent to the
    $H^1(\domain)$-scalar product which leads to 
    a one-to-one correspondence between test functions $\vphi$
    and linear functionals on $H^1(\domain)$.
    Since $A\vphi \in L^{q/(q-1)}(\domain)$,
    the right hand side tends to zero for any $\vphi \in H^1(\domain)$
    as $n\to\infty$.
    We hence see that 
    \begin{equation*}
	A_n \rightharpoonup A \quad\mbox{ in } H^1(\domain).
    \end{equation*}
    It follows from continuity of the trace mapping $H^1(\domain) \to
H^{1/2}(\bndry)$
    and the compact embedding $H^{1/2}(\bndry) \hookrightarrow\hookrightarrow
L^2(\bndry)$, 
    that
    \begin{equation*}\
	A_n \to A \quad \mbox{ in } L^2(\bndry).
    \end{equation*}
\end{proof}


The differentiation of \refe{eq:MITweak} at $\sigma$ in the direction $h$ yields
\begin{equation}\label{eq:MITsensitivity}
  (\muinv\nabla \delta A, \nabla \vphi) + (i\omega\sigma \delta A, \vphi) =
-(i\omega hA, \vphi)
  \qquad \forall \vphi \in H^1(\domain).
\end{equation}
The symbol $\delta A:=\delta A(\sigma;h)$ stands for 
the \emph{variation} (G\^ateaux differential) of $A=A(\sigma)$ in the direction
$h$.
The variation $\delta A$ is sometimes called the sensitivity of $A$
and \refe{eq:MITsensitivity} the \emph{sensitivity equation},
which is a well-posed problem with the unique solution $\delta A$  
for any $h$ from $L^q(\domain), q>1$. It is straightforward to verify 
that for given $\sigma$ the mapping $ h \mapsto \delta A(\sigma; h)|_\bndry$ is
linear 
and bounded operator in $L(M, L^2(\bndry))$. Recalling the relationship between
the variation and Fr\'echet derivative, we see that $\im$ is Fr\'echet
differentiable at $\sigma$ and 
\begin{equation*}
    \im'(\sigma)h=\delta A(\sigma; h)|_\bndry.
\end{equation*}
The map $\im':M \to L(M, L^2(\bndry))$ is continuous in $\sigma$ by the similar
reasoning
as in the proof of Theorem \ref{lem:impedance1}
and so we have the following assertion.
\begin{lemma}\label{lem:impedance2}
The impedance map $\im: M\to L^2(\bndry)$ is $C^1$-Fr\'echet differentiable.
\end{lemma}


\subsection{Inverse problem}

By the inverse problem in MIT we will understand the reconstruction of the
piecewise-constant conductivity $\sigma$ in the imaged body based on a finite
number of
Dirichlet-to-Neumann data $(e,m)$ corresponding to the impedance map
\refe{eq:NtDmap}.
The boundary data $m$ are essentially voltage measurements associated with
excitations $e$.
Lemma \ref{lem:impedance1} implies that $\im$ is a compact operator
and so the recovery of $\sigma$ is inherently an ill-posed problem.

We employ the topology-to-shape continuation method (TSCM) from Section
\ref{sec:tscm} to solve MIT. We look for the conductivity $\sigma$ in the form
\refe{eq:sigma_tscm}, i.e.
\begin{equation*}
    \sigma = (1-\lambda)\sldva + \lambda\spc,
\end{equation*}
where $\spc$ is a double-valued piecewise constant function as it is considered
in Section \ref{sec:BM} for the example of electrical impedance tomography.
The associated continuation Tikhonov functional for MIT read as
\begin{equation}\label{eq:MITfunctional}
    \tab(\sigma) = \fid(\sigma) \,+ \ruw(\spc, \sigma_{L_2}, \lambda),
\end{equation}
where $\fid$ is the fidelity term
\begin{equation}\label{eq:MITfidelity}
    \fid(\sigma) = \int_\bndry |\im(\sigma, \omega, e)-m|^2 \dS.
\end{equation}
The regularization part $\ruw$ is given by 
\begin{equation}\label{eq:ruw_mit}
 \eqalign{
    \ruw(\spc, \sigma_{L_2}, \lambda) =& \lambda\alpha\Big[ \int_\domain
\sqrt{|\spc|^2+\varepsilon}\dx 
    + J_{\varepsilon}(\spc)\Big]\\%
    &+ (1-\lambda)\beta\norma{\sldva}^2,
}
\end{equation}
which complies with the TSCM analysis in Section \ref{sec:tscm}. The forward
problem operator $\Lambda$ of MIT is an admissible operator fulfilling 
assumption (A2) of Section \ref{sec:framework} as it is shown in Section
\ref{sec:mitforward}.
Altogether, the theory of Section \ref{sec:framework} is applicable to the
inverse problem of MIT as stated in this section.

\subsubsection{Adjoint problem}

In Section \ref{sec:numer_exp} we will use a gradient-based method (the steepest
descent method) to find a minimizer of \refe{eq:MITfunctional}. 
Let us express the derivative of fidelity term \refe{eq:MITfidelity} using an
\emph{adjoint variable}. The variation of $\fid$ in the direction $h$ reads as
\begin{equation*}
    \eqalign{
    \delta \fid(\sigma;h)
    &=\lim_{t\to0}\frac{\fid(\sigma+th)-\fid(\sigma)}{t}\\
    &= (\im(\sigma)-m, \delta \im(\sigma;h))_\bndry + (\delta \im(\sigma;h),
    \im(\sigma)-m)_\bndry\\ 
    &= 2\Re\left[(\delta \im(\sigma;h), \im(\sigma)-m)_\bndry\right],
    }
\end{equation*}
where the variation $\delta\im(\sigma;h)\equiv\delta A$ solves 
the sensitivity equation \refe{eq:MITsensitivity}.
We now introduce the adjoint variable $Z$ which satisfies
\begin{equation}\label{eq:MITadjoint}
    (\muinv\nabla \vphi, \nabla Z) + (i\omega\sigma \vphi, Z) = -(\vphi,
    \im(\sigma)-m)_\bndry
    \qquad \forall \vphi \mbox{ in }H^1(\domain),
\end{equation}
to establish that
\begin{equation}\label{eq:MITfunctional_derivative}
    \eqalign{
    \delta \fid(\sigma;h) 
    &= 2\Re\left[(\delta \im(\sigma;h), \im(\sigma)-m)_\bndry\right]\\
    &\mathop{=}\limits^{\refe{eq:MITadjoint}} 2\Re\left[-(\muinv\nabla \delta A,
\nabla Z) - (i\omega\sigma \delta A,
Z)\right]\\
    &\mathop{=}\limits^{\refe{eq:MITsensitivity}} 2\Re\left[ (i\omega hA,
Z)\right].}
\end{equation}

Let us note, that the variational problem \refe{eq:MITadjoint} for $Z$ is
uniquely solvable given the properties of the material parameters and of the
impedance map $\Lambda.$ We assume that $m\in L^2(\bndry).$


\subsection{Implementation of TSCM method}\label{sec:tscm_implementation}

In this section we describe the implementation of the topology-to-shape
continuation
method (TSCM) for the problem of the magnetic induction tomography.

The practical implementation of the TSCM algorithm presented in Algorithm
\ref{alg:1}
closely follows the theoretical exposition. The outer loop successively
increases the value of $\lambda$ by the increment $\Delta\lambda$ starting from
$\lambda = 0$. It terminates when $\lambda = 1$ is reached. The number of steps
is determined by $\Delta\lambda.$ The inner loop constitute
more or less a standard adjoint-variable based steepest descent algorithm 
for minimization of \refe{eq:MITfunctional} for the fixed $\lambda$. The
number $n$ stands for the total number of iterations through both loops in
Algorithm \ref{alg:1}.

\begin{Algorithm}[ht]
\small
\begin{algorithm}[H]
	\SetKwRepeat{dowhile}{do}{while}
	\KwData{$n=0;\; \lambda=0;\;
	\sigma_n = \sldva{_{,n}} = \delta_1;\;
	\phi_n = -\delta_2;$}
	\SetKwBlock{MyBlock}{}{}
	\dowhile{$\lambda < 1$}
	{
	    $s_n = 2$\; 
	    \dowhile{$|\nabla_{\sldva}\tab{_{,n}}|^2 +
|\nabla_{\phi}\tab{_{,n}}|^2 > \tau_1^2$ {\rm and} $s_n>\tau_2$}
	    {
	    \emph{Compute the derivatives:}
	    \MyBlock{
	    $\sigma_n \longrightarrow \mbox{direct problem }\refe{eq:MITweak}
\longrightarrow A_n$\;
	    $(\sigma_n, A_n)  \longrightarrow \mbox{adjoint problem
}\refe{eq:MITadjoint} \longrightarrow Z_n$\;
	    $(A_n, Z_n)  \longrightarrow \mbox{cost functional derivative }
\refe{eq:MITfunctional_derivative} \longrightarrow \nabla_\sigma \fid_n$\;
	    $\nabla_\sigma\fid_n + \refe{eq:nsldvafid} + \refe{eq:nsldvaruw}
\longrightarrow \nabla_{\sldva} \tab{_{,n}}$\;
	    $\nabla_\sigma\fid_n + \refe{eq:nphifid} + \refe{eq:nphiruw}
\longrightarrow \nabla_{\phi} \tab{_{,n}}$\;}
	    \emph{Find the optimal step:}
	    \MyBlock{
	    $s_n = \mbox{Linesearch}(\sigma_n,\nabla_{\sldva}\tab{_{,n}},$
$\nabla_{\phi}\tab{_{,n}})$\;}
	    \vspace{-0.5cm}
	    \emph{Update $\sigma_n$:}
	    \MyBlock{
	    $\sldva{_{, n+1}} = \sldva{_{, n}}-s_n \nabla_{\sldva}
\tab{_{,n}}$\;
	    $\phi_{n+1} = \phi_n-s_n\nabla_{\phi} \tab{_{,n}}$\;
	    $\sigma_{n+1} = \lambda\spc(\phi_{n+1}) + (1-\lambda)\sldva{_{,
n+1}}$\;}
	    $n =n+1$\;
	    }
	    $\lambda = \lambda + \Delta\lambda$\;
	}
\end{algorithm}
\normalsize
\caption{Topology-to-shape continuation algorithm \label{alg:1}}
\end{Algorithm}

We use the level set method  \cite{Osher1988} to parametrize the
conductivity $\spc$ introduced in \refe{eq:parameter_PC}. One first
defines the level set function $\phi$ for the subset $D \subset \domain$ with
its boundary $\partial D$
\begin{equation*}
  \phi(x) = \left\{ 
  \begin{array}{l l }
    \mbox{distance}(x, \partial D) & x \in D,\\
    -\mbox{distance}(x, \partial D) & x \in \domain/D.\\
  \end{array} \right.
\end{equation*}
The zero level set of $\phi$ represents the boundary of $D$ (its "interface").
The piecewise-constant conductivity $\spc$ is then parametrized as
\begin{equation}\label{eq:sigmaLS}
\spc(\phi) = \sigma_1 H(\phi) + \sigma_2 (1-H(\phi)),
\end{equation}
where $H$ stands for the unit step Heaviside function.
We use the following smooth approximations of $H$ and its derivative:
\begin{equation}\label{eq:Heps}
  H_\varepsilon(\phi) = \frac{1}{\pi}\arctan\frac{\phi}{\varepsilon} +
\half,\qquad 
H'_\varepsilon(\phi)=\delta_\varepsilon(\phi)=\frac{\varepsilon}{
\pi(\phi^2+\varepsilon^2)}.
\end{equation}

The gradient of \refe{eq:ruw_mit} with respect to $\spc$ is evaluated as the
solution of the variational problem
\begin{equation}\label{eq:BVprojection}
\scal{\nabla_{\spc}\ruw}{h} = \lambda\alpha\left[\scal{\frac{\nabla
\spc}{\sqrt{|\nabla \spc|^2 +\varepsilon}}}{\nabla h} +\scal{\spc}{h}\right]
\end{equation}
for all $h\in H_0^1(\Omega)$. It is, in fact, a projection of 
$\partial_{\spc}\ruw$ onto the nodes of the finite element mesh.
We remark that all the variational problems (\refe{eq:MITweak},
\refe{eq:MITadjoint} etc.) are solved by finite element method where
$H^1(\Omega)$ is approximated by
linear Lagrange basis functions. Using \refe{eq:sigmaLS} together with
\refe{eq:Heps} we have 
\begin{equation}\label{eq:nphiruw}
\nabla_{\phi}\ruw = (\sigma_1 -\sigma_2)H'_\epsilon(\phi)\nabla_{\spc}\ruw.
\end{equation}
The gradient of \refe{eq:ruw_mit} with respect to $\sldva$ is simply
\begin{equation}\label{eq:nsldvaruw}
\nabla_{\sldva}\ruw = 2(1-\lambda)\beta\sldva.
\end{equation}
The gradient $\nabla_\sigma \fid$ of the fidelity term $\fid$ with respect to
$\sigma$ is evaluated from
\refe{eq:MITfunctional_derivative} again by projection onto the nodes of the
finite element mesh as in \refe{eq:BVprojection}:
\begin{equation*}
\nabla_\sigma \fid = 2\Re[i\omega AZ].
\end{equation*}
This yields 
\begin{equation}\label{eq:nsldvafid}
\nabla_{\sldva}\fid = (1-\lambda)\nabla_\sigma \fid
\end{equation}
and 
\begin{equation}\label{eq:nphifid}
\nabla_{\phi}\fid = \lambda(\sigma_1
-\sigma_2)H'_\epsilon(\phi)\nabla_\sigma\fid.
\end{equation}

We do not optimize with respect to the constants $\sigma_1$ and $\sigma_2$,
which we consider to be known. However, Algorithm \ref{alg:1} is easily
extendable to the
case of unknown $\sigma_1$ and $\sigma_2$.

We emphasize that we do not assume any a priori knowledge about the shape of
$D.$ The unknowns $\phi$ and $\sldva$ are initiated as $\phi=-\delta_1$ and
$\sldva = \delta_2$ with $\delta_1$ and $\delta_2$ being some positive
constants, $\delta_2\approx\smin$. It means that initially ($\lambda=0$) the
whole domain $\Omega$ is occupied by a weak phase. In addition  we have zero
inclusion $D$ and thus the value of $\spc$ is $\sigma_2$ in the whole domain.

In Algorithm \ref{alg:1} the search for an optimal step-size $s_n$ might be the
most
time-consuming part, since the Linesearch-algorithm detects the optimal $s_n$ by
the evaluation of the cost functional for different intermediate values of $s_n$
and one such evaluation means to solve one forward problem \ref{eq:MITweak}.
However, we do not need to find the optimal value of $s_n$ for which the drop of
$\tab$ is maximal. It is enough to find one value for which $\tab$ drops
sufficiently (the method is then no more steepest descent). We update $s$
according to the following simple rule \cite{Cimrak2010}:
\[
s_{n+1} = 2s_n \mbox{ if } \tab(\sigma_{n}(s_{n−2})) < \tab (\sigma_{n−1}),
\]
i.e. when $s_{n−1} := s_{n−2}$ gave a reduction of cost functional value, we try
double the step. If in the next step $s_n$ does not give a descent, we take the
step with the smallest $k$ from the sequence $s^k_n = s^{k-1}_n/2,$  $k =
1,\dots,\infty$ such that we have descent.
The last part is the actual update process. The inner cycle of Algorithm
\ref{alg:1} stops
when the norm of gradient is sufficiently small ($\le \tau_1$) or the computed
gradient is not a descent direction anymore, i.e. $s_n \le \tau_2$, where
$\tau_1$ and $\tau_2$ are suitable constants.


\subsection{Numerical experiments}\label{sec:numer_exp}

In all the experiments we use synthetic data. The number $N$ of the
measurements for every experiment corresponds to
the number of excitation coils $N(e)$ (see Figure \ref{fig:MIT}) multiplied
with the number of excitation frequencies $N(\omega)$. The fidelity functional
reads
\begin{equation}\label{eq:MITfidelity2}
    \fid(\sigma) = \sum_{\omega}\sum_{e}\int_\bndry |\im(\sigma, \omega, e)-m|^2
\dS.
\end{equation}

We take $\sigma_1 = 20 S\cdot m^{-1}$ and $\sigma_2 = 2 S\cdot m^{-1}$ and
$\mu=\mu_0$ which complies with physiological conditions. For comparison, in
non-destructive testing of metallic pieces normal magnitudes of $\sigma$ are in
millions of $S\cdot m^{-1}$ and $\mu \gg \mu_0.$

All the excitation currents $e_i = 1 A\cdot m^{-1},$ $i=1,\dots, N(e)$. The
angular excitation frequencies $\omega_i = 2\pi f_i = 2\pi 2^{15+i},$
$i=0,\dots, N(\omega)-1.$ The basic frequency $f_0 = 2^{15}$ is set so that
$\mu^{-1}>\omega_0 \max(\sigma_1, \sigma_2).$ For such a base frequency the
magnetic phenomena dominate the electric ones.

The parameters in Algorithm \ref{alg:1} are $\tau_1 = 10^{-5}$, $\tau_1 =
10^{-6},$ $\delta_1 = 1$, $\delta_2 = 0.01.$ We implemented the algorithm in
FreeFem++ \cite{Hecht2009}. In all the experiments for both $\sldva$ and
$\phi$ we use identical fixed regular meshes with homogeneous division of the
boundary $\Gamma$. We also always consider $28$ excitation coils, i.e.
$N(e)=28,$ and the regularization parameters $\alpha$ and $\beta$ are fixed as
$\alpha=\beta=0.00001.$  In \refe{eq:Heps} we take $\epsilon = h^2, $ where $h$
is the diameter of the finite element mesh. If not stated otherwise we take
$\Delta\lambda=0.1.$

\begin{figure}
\begin{center}
  \subfigure[LSM, no noise, initial
$\phi_0$]{\includegraphics[width=0.35\columnwidth]{%
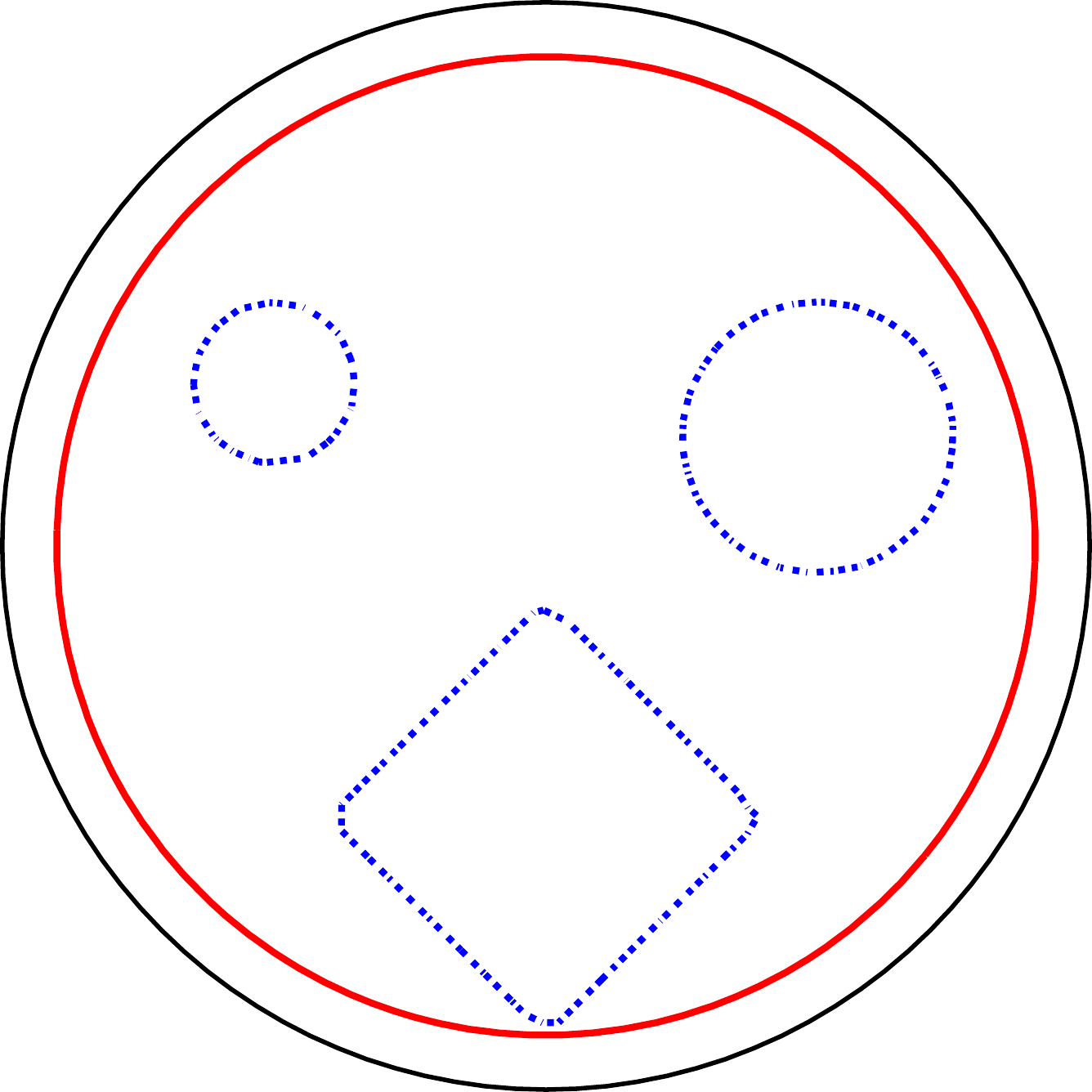}}
  \hspace{1cm}
  \subfigure[LSM, no noise, final $\phi_{37}$]{%
\includegraphics[width=0.35\columnwidth]{%
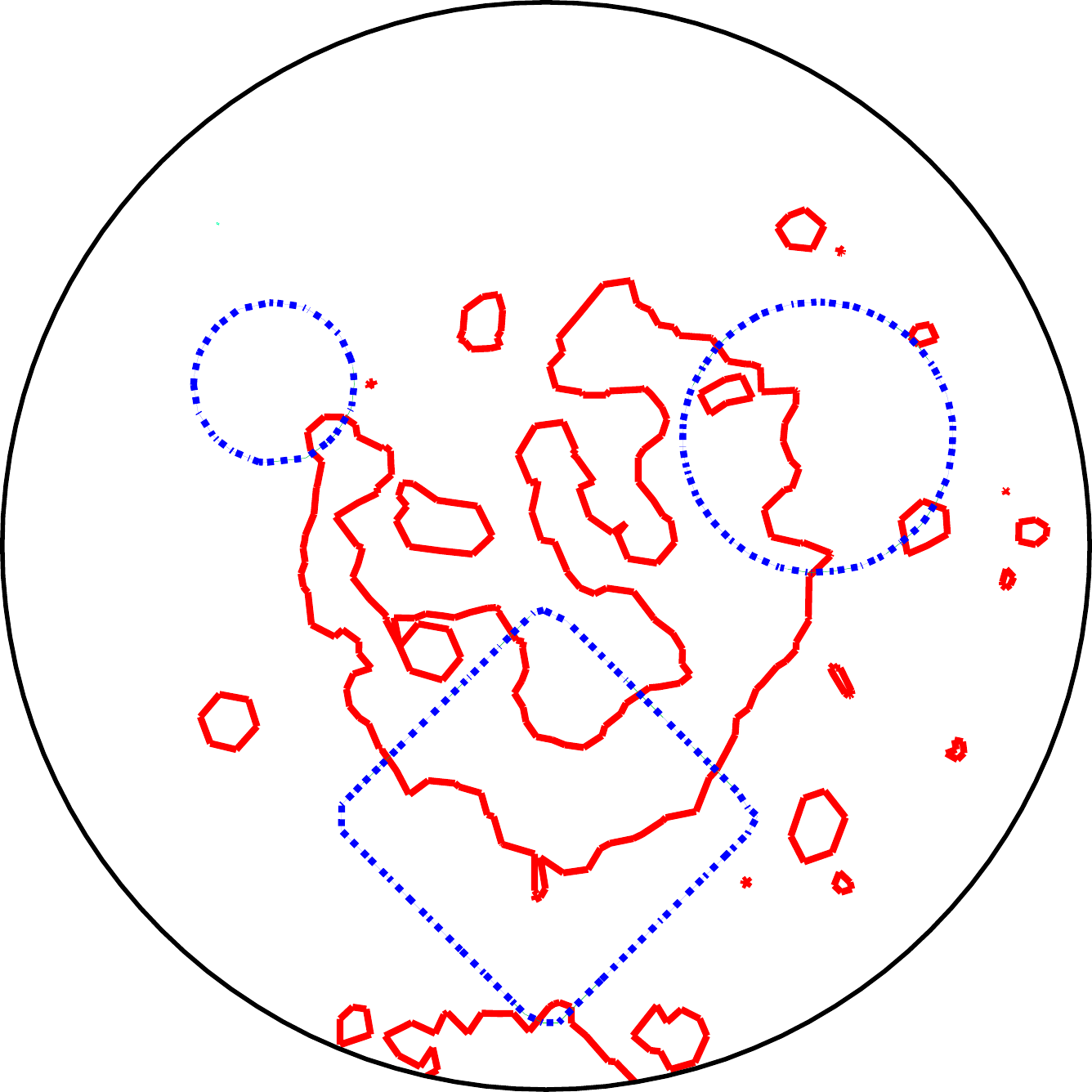}}
  \subfigure[TSCM, no noise, final $\phi_{228}$]{%
\includegraphics[width=0.35\columnwidth]{%
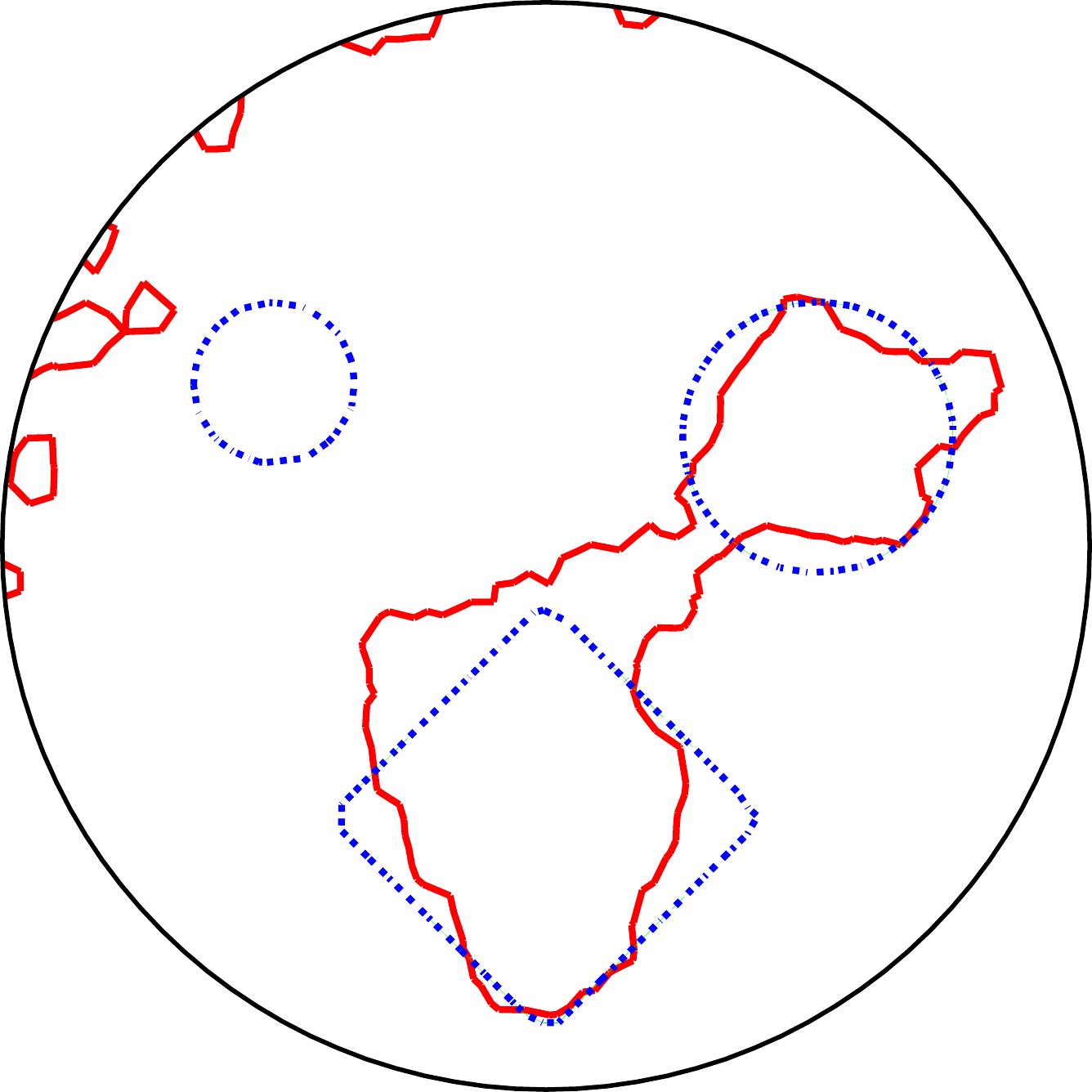}}
  \hspace{1cm}
  \subfigure[TSCM, $1\%$ noise, final $\phi_{188}$]{%
  \includegraphics[width=0.35\columnwidth]{%
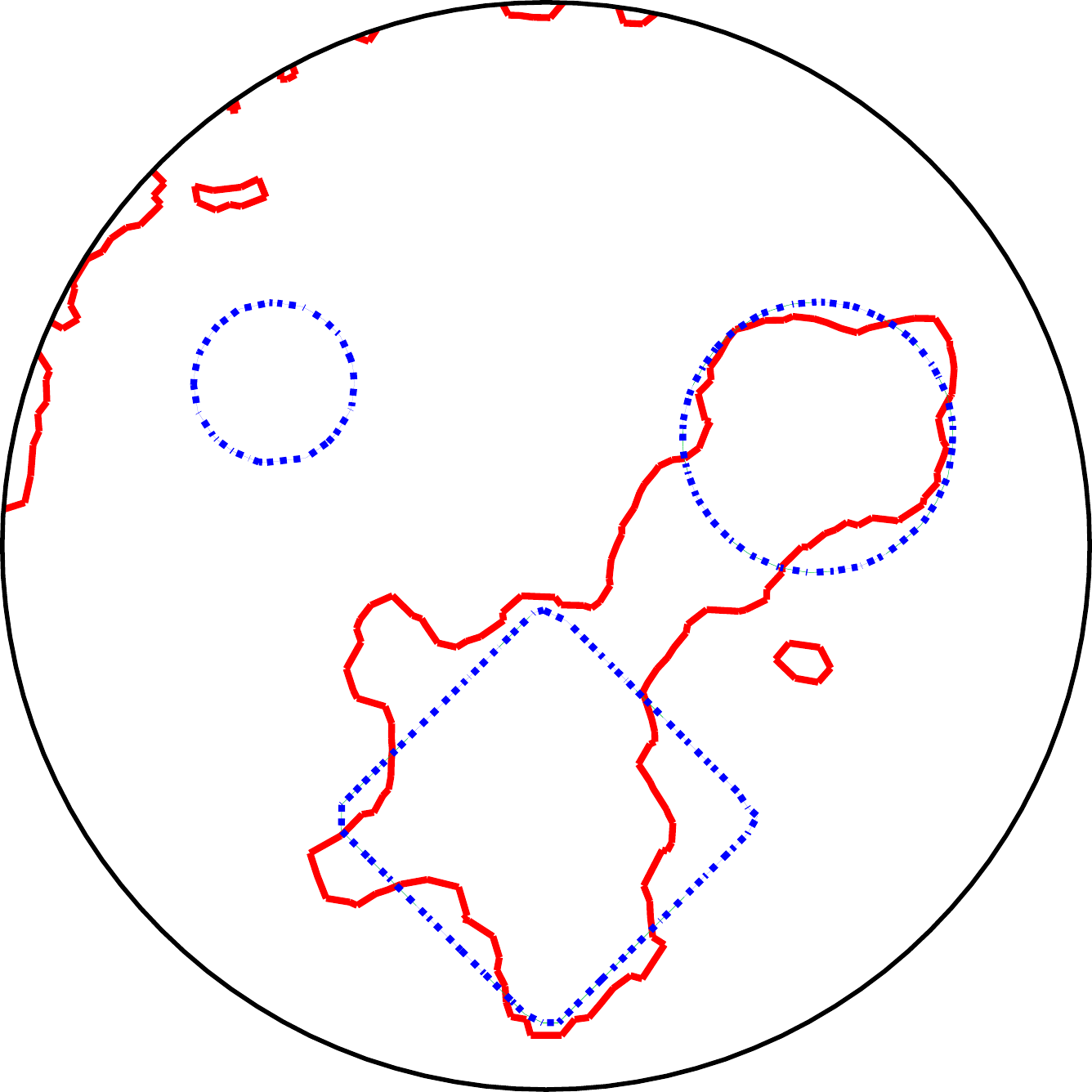}}
  \caption{Comparision between the standard LSM and
TSCM{\label{fig:tscm_vs_lsm}}}
\end{center}
\end{figure}

We first compare the performance of the continuation algorithm (TSCM) and
the standard level set method (LSM) on an example with a
non-trivial topology (Figure \ref{fig:tscm_vs_lsm}). The
blue dotted line represents in all the figures the exact phantom and the red
line is the numerical approximation. The initial shape of $\spc$ for the
standard LSM is depicted in Figure \ref{fig:tscm_vs_lsm}(a). Figure
\ref{fig:tscm_vs_lsm} displays the results for
the base angular frequency $\omega_0$. The LSM in Figure
\ref{fig:tscm_vs_lsm}(a) ended up in a local minimum after $37$ iterations. The
algorithm stopped because
the computed gradient was not a descent direction anymore, i.e. $s_{37} <
\tau_2$. We see that without a proper initial guess, the standard LSM failed to
recover the desired shape. On the other hand, the TSCM in Figure
\ref{fig:tscm_vs_lsm}(c) for zero noise provided a decent
approximation. Both bigger phantoms are recovered quite successfully but
they stay connected. The smallest phantom is not identified properly. Only
certain allocation of its mass is identified along the proximal boundary.
Even for $1\%$ noise the TSCM method provided a decent approximation (Figure
\ref{fig:tscm_vs_lsm}(d)). The method seems to be rather stable with respect to
noise. We recall, that the standard LSM is very sensitive when only boundary
measurements are available, e.g. in \cite[Figure 7]{Chung2005} only a
noise level of $0.01\%$ is considered in a case of a complicated phantom for the
problem of electric impedance tomography.

We next perform numerical experiments that use explicit dependency of MIT model
on the frequency $\omega$. The results are presented in Figure
\ref{fig:multifreq3balls} for the phantom identical to the previous
single-frequency experiment in Figure \ref{fig:tscm_vs_lsm}. We consider the
four-frequency case $N(\omega)= 4$ and four levels of noise: $1\%$, $5\%$,
$10\%$ and $20\%$. The blue line is again the exact shape and the red line is
its TSCM-identification. As expected we got
more accurate recovery of $\spc$. For the noise levels up to $10\%$ all the
components of the phantom are quite accurately identified, accuracy gradually
decreasing. Even for noise level of $20\%,$ the identification is surprisingly
accurate and all the components are identified, however two bigger components
stay connected by a bridge. This experiment confirms our
conjecture that the method is very stable with respect to the non-systematic
noise.

\begin{figure}[ht]
\begin{center}
  \subfigure[$1\%$ noise, final
$\phi_{454}$]{\includegraphics[width=0.35\columnwidth]{%
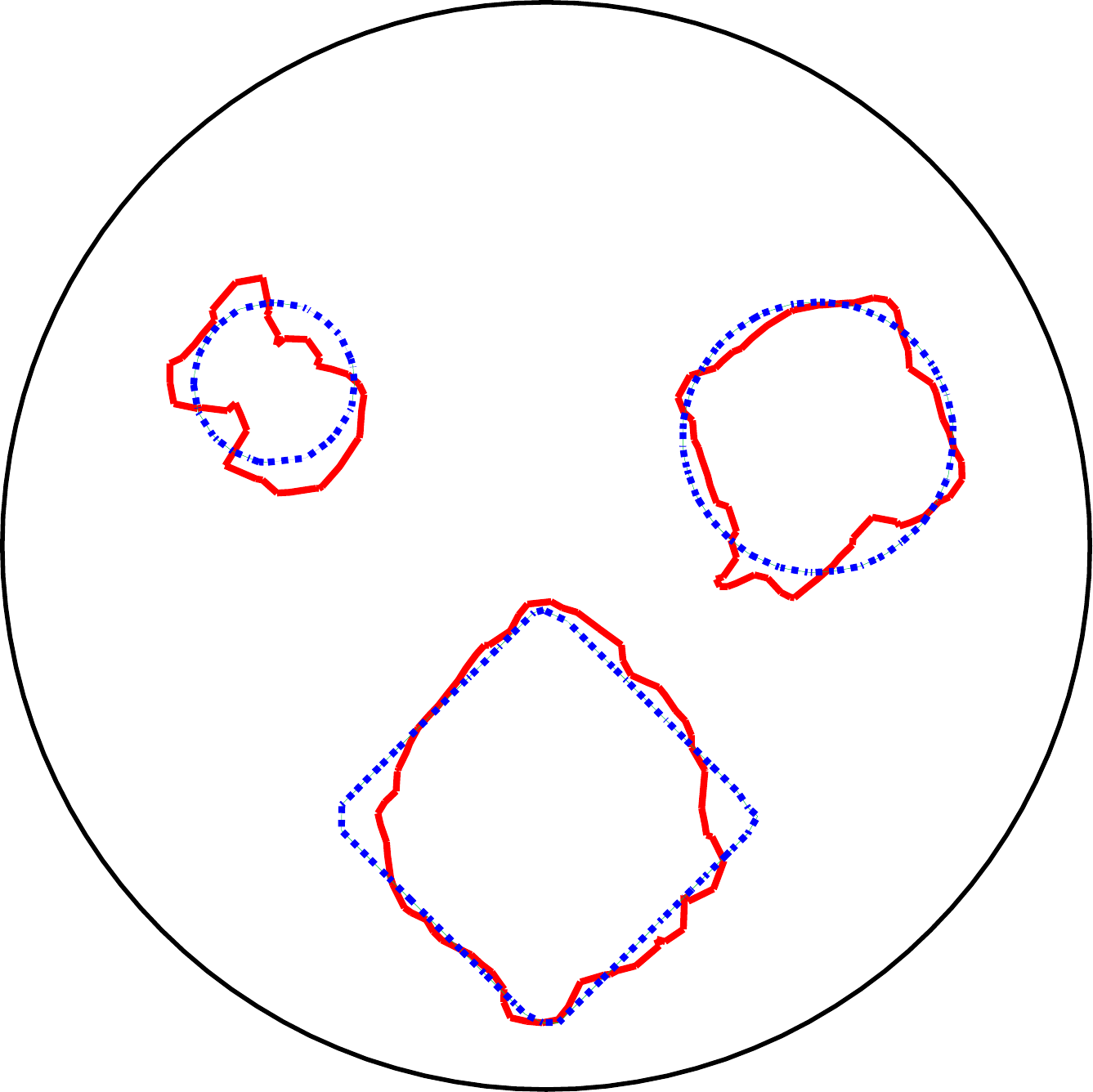}}
  \hspace{1cm}
  \subfigure[$5\%$ noise, final
$\phi_{350}$]{\includegraphics[width=0.35\columnwidth]{%
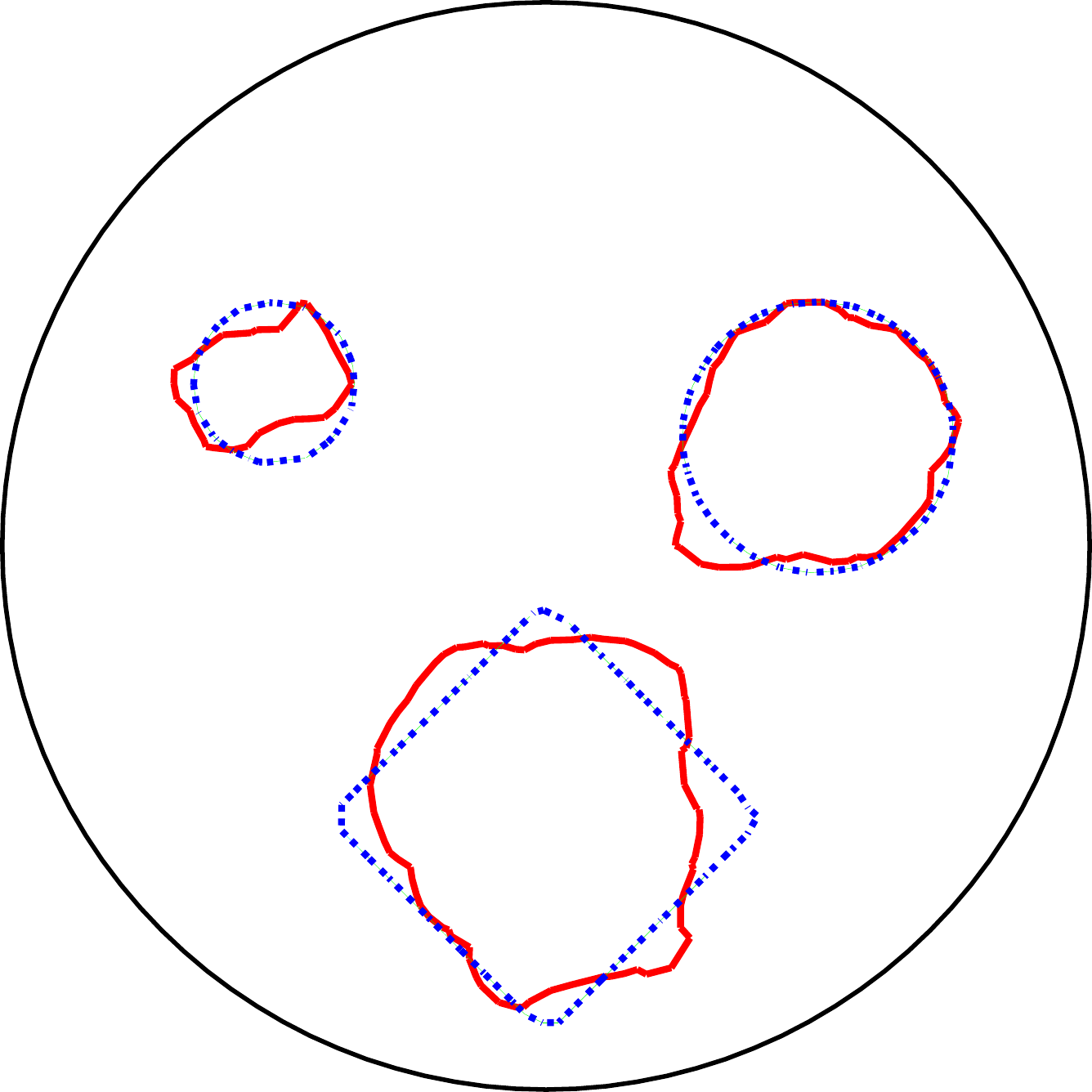}}
  \subfigure[$10\%$ noise, final
$\phi_{290}$]{\includegraphics[width=0.35\columnwidth]{%
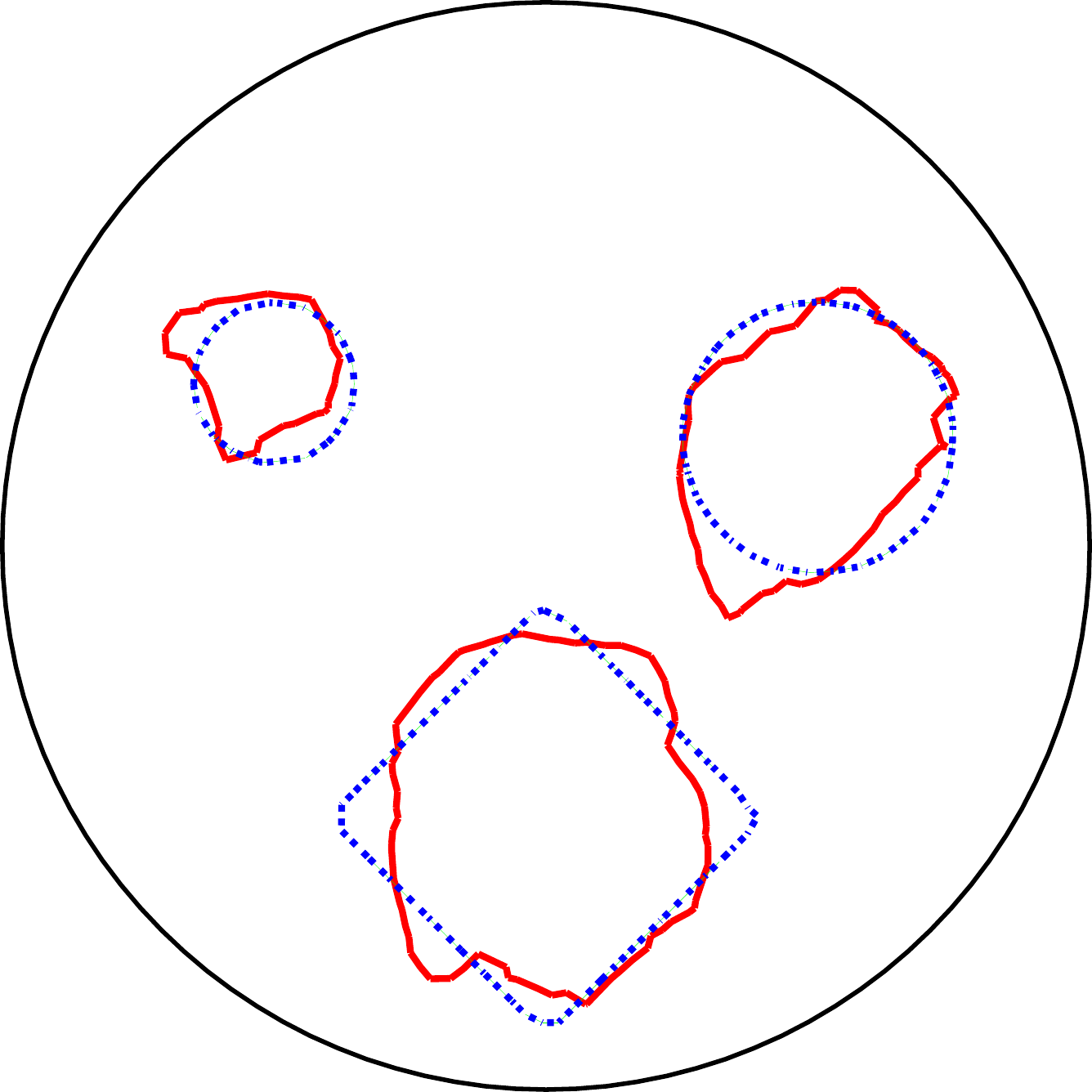}}
  \hspace{1cm}
  \subfigure[$20\%$ noise, final
$\phi_{386}$]{\includegraphics[width=0.35\columnwidth]{%
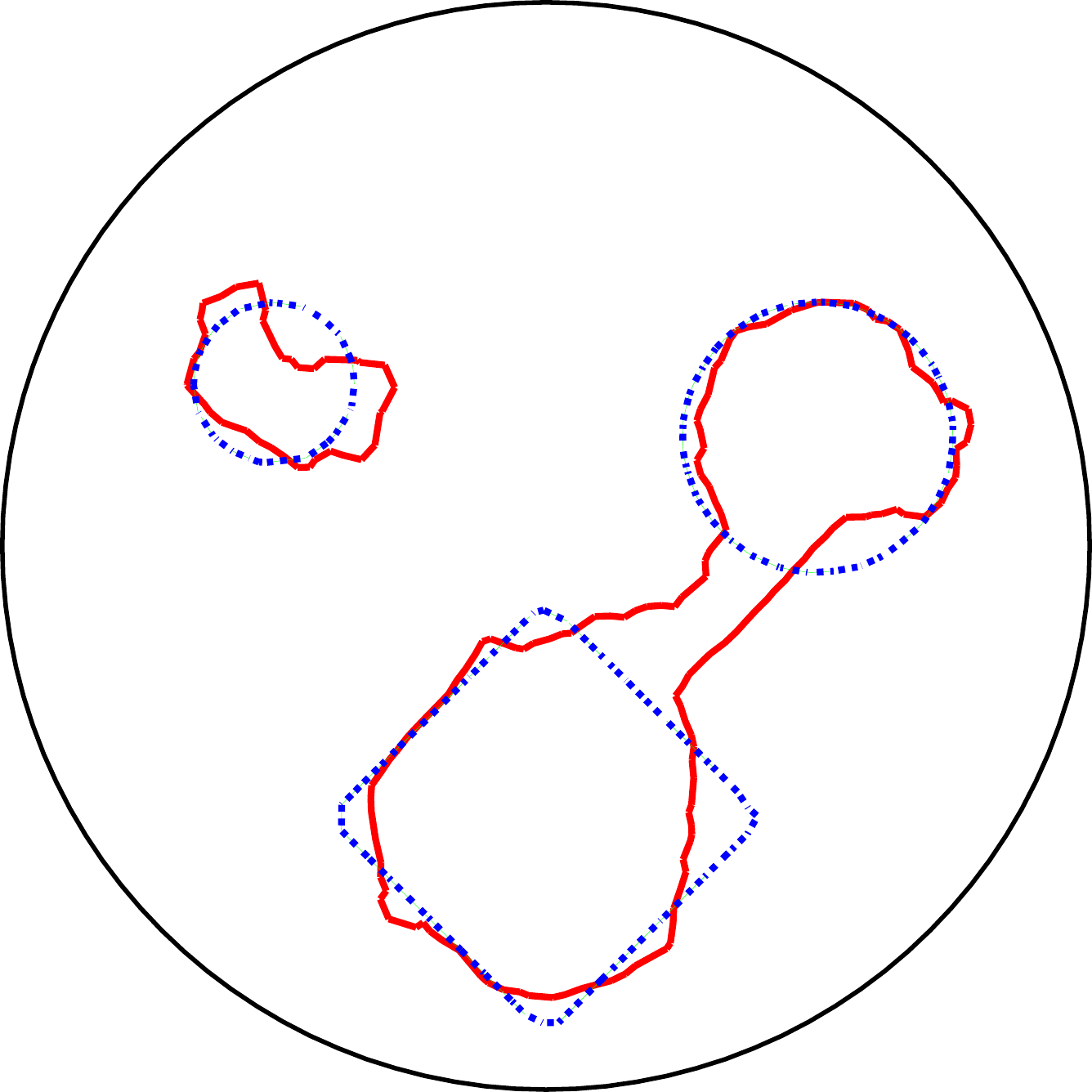}}
  \caption{TSCM: multiple frequency case %
$N(\omega)=4${\label{fig:multifreq3balls}}}
\end{center}
\end{figure}

Noise causes non-convexity of the fidelity term $\fid$ regardless the
properties of the forward operator $\F$. Provided the data contain sufficient
information to identify the phantom, the TSCM is able to eliminate this type of
non-convexity. We are convinced the reason lies within the nature of the method.
The TSCM is essentially a convexification approach.

The convergences of the fidelity term $\fid(\sigma)$ and of the relative error
between the computed conductivity $\sigma_{TSCM}$ and exact conductivity
$\sigma_{exact}$
\begin{equation}\label{eg:relerrorsigma}
e(\sigma) = \frac{\norm{\sigma_{TSCM} -
\sigma_{exact}}{\Ldva}}{\norm{\sigma_{exact}}{\Ldva}}
\end{equation}
with respect to the total number of iterations $n$ of Algorithm \ref{alg:1} are
depicted in Figure \ref{fig:minim}(a). These graphs correspond to the
experiment 
of Figure \ref{fig:multifreq3balls}(a). The distribution of the number of
iterations for different $\lambda-$steps is depicted in Figure
\ref{fig:minim}(b). In general, the first iteration of the TSCM for $\lambda=0$
is the most time consuming, which is natural, because it is nothing else than
the minimization of $\tab$ in the space $L^2(\domain)$. It provides the
information about ``the optimal topology'' for $\spc.$ Once this good initial
guess is found, the continuation method rather quickly transforms this function
to the desired piecewise-constant conductivity $\spc$.

\begin{figure}[ht]
\begin{center}
  \subfigure[Convergence]{\includegraphics[width=0.45\columnwidth]{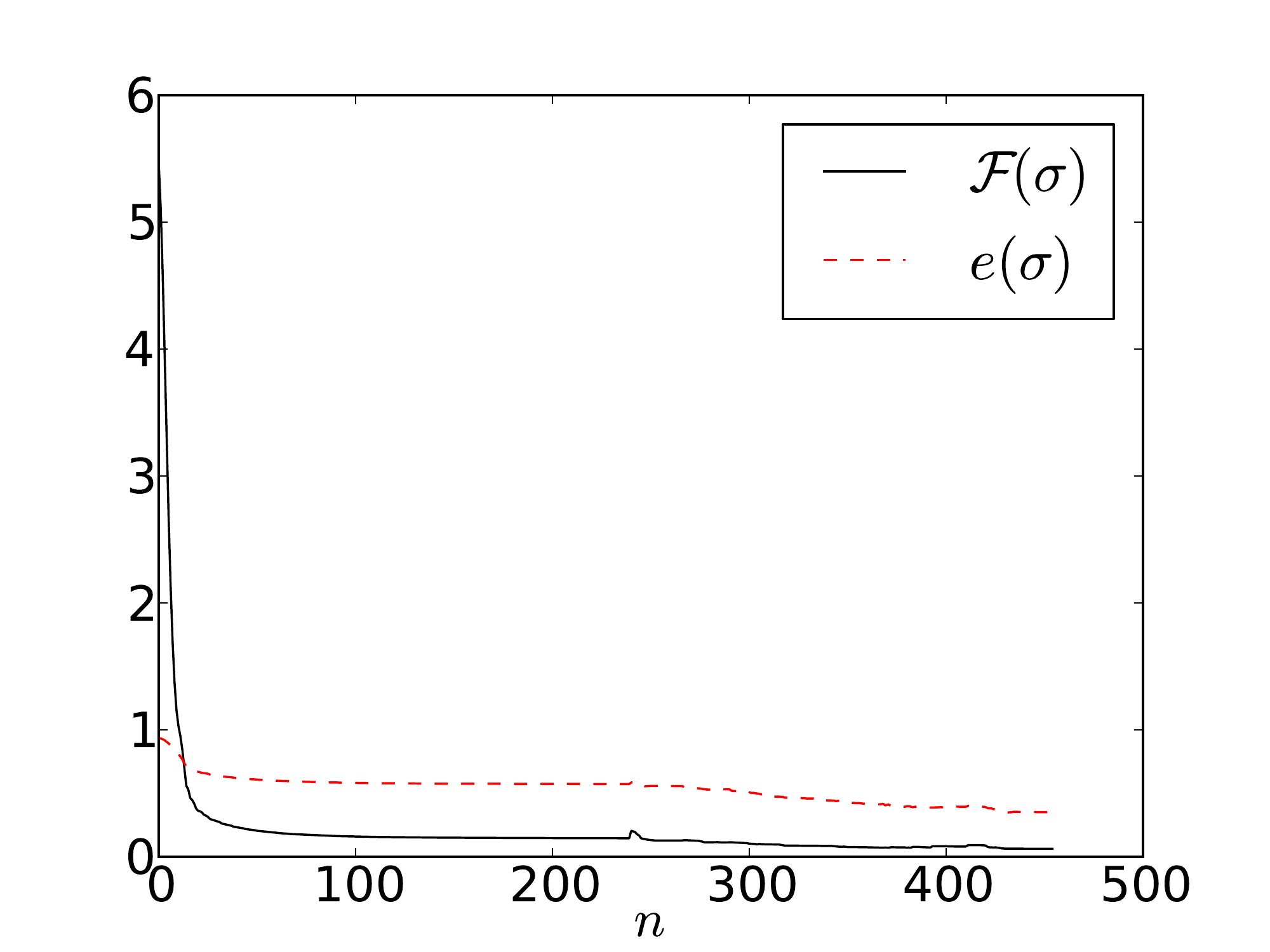}}
  \subfigure[Distribution of
iterations]{\includegraphics[width=0.45\columnwidth]{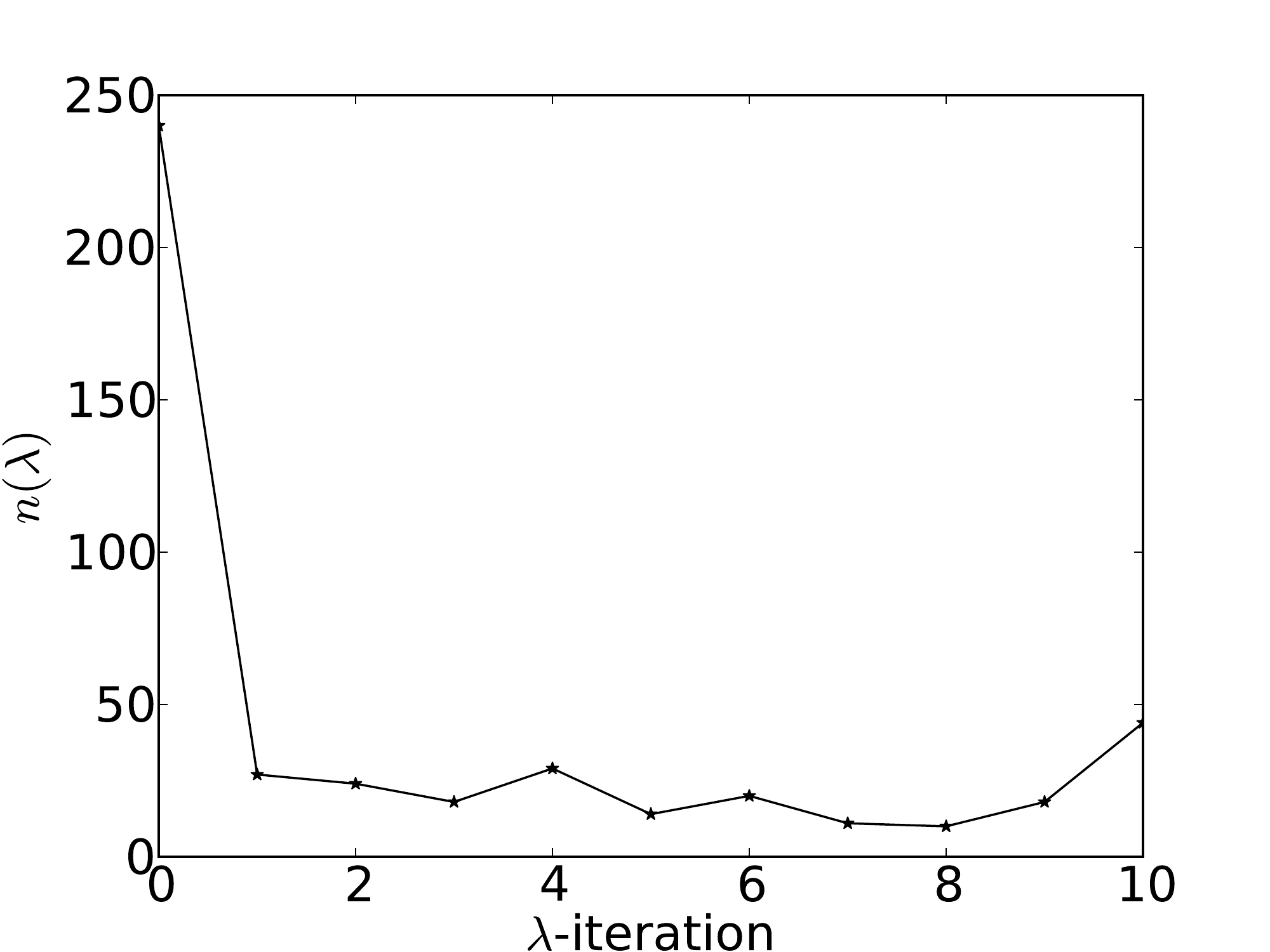}}
  \caption{TSCM: experiment 1; $\rho = 1\%$; $N(\omega) = 4${\label{fig:minim}}}
\end{center}
\end{figure}

Next, we consider a more complicated phantom with its two components
touching and one of them being a torus. We again consider four excitation
frequencies $N(\omega)=4.$ The results for two level of noise, $1\%$ and $10\%$,
are depicted in Figure \ref{fig:multifreqTorus}. Again, we obtained a decent
reconstruction even for $10\%$ noise. Except the outside boundary also the hole
of the torus is well identified. The less resolved regions are those where
the components are touching and the center of the domain.

\begin{figure}[ht]
\begin{center}
  \subfigure[$1\%$ noise, final
$\phi_{533}$]{\includegraphics[width=0.35\columnwidth]{%
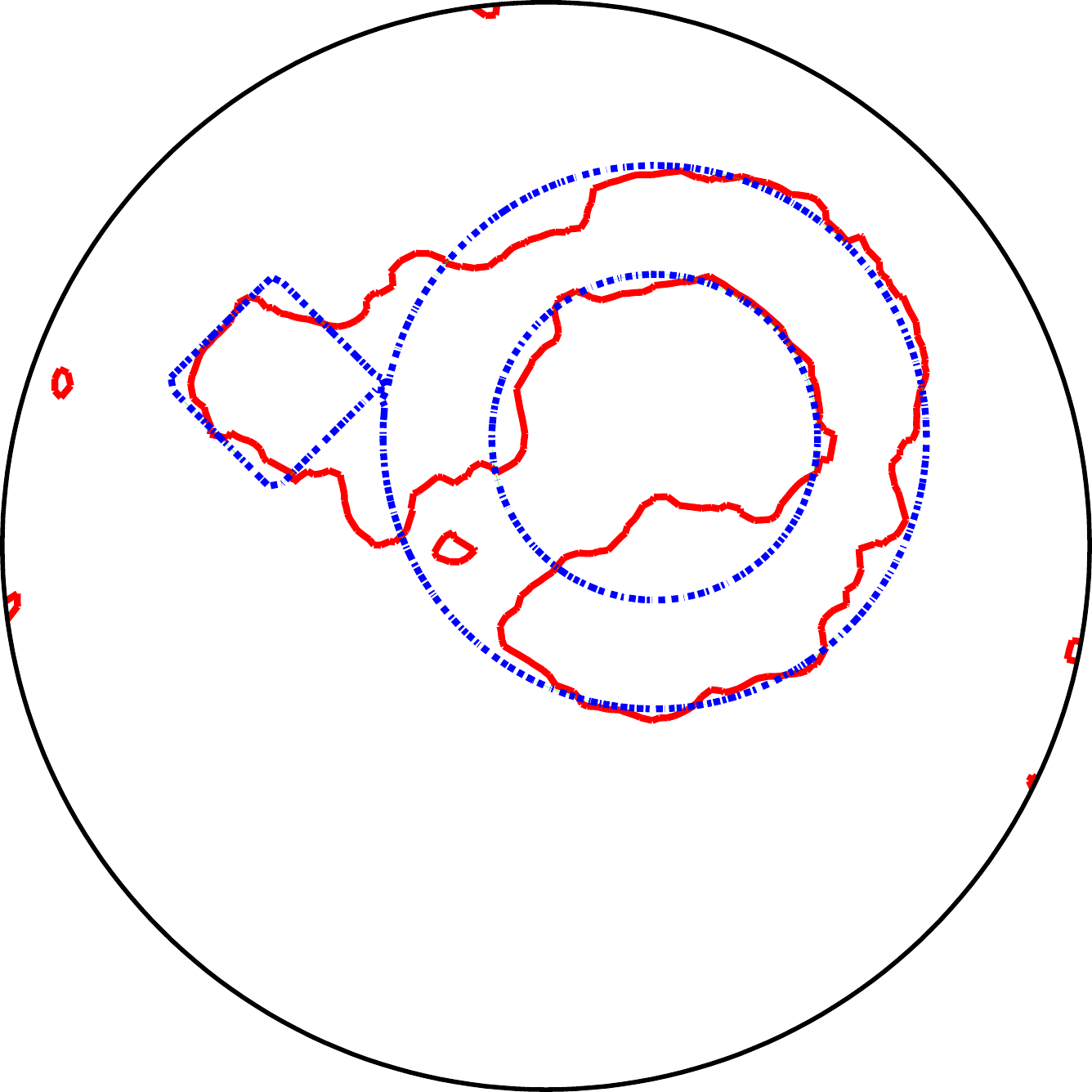}}
  \hspace{1cm}
  \subfigure[$10\%$ noise, final
$\phi_{202}$]{\includegraphics[width=0.35\columnwidth]{%
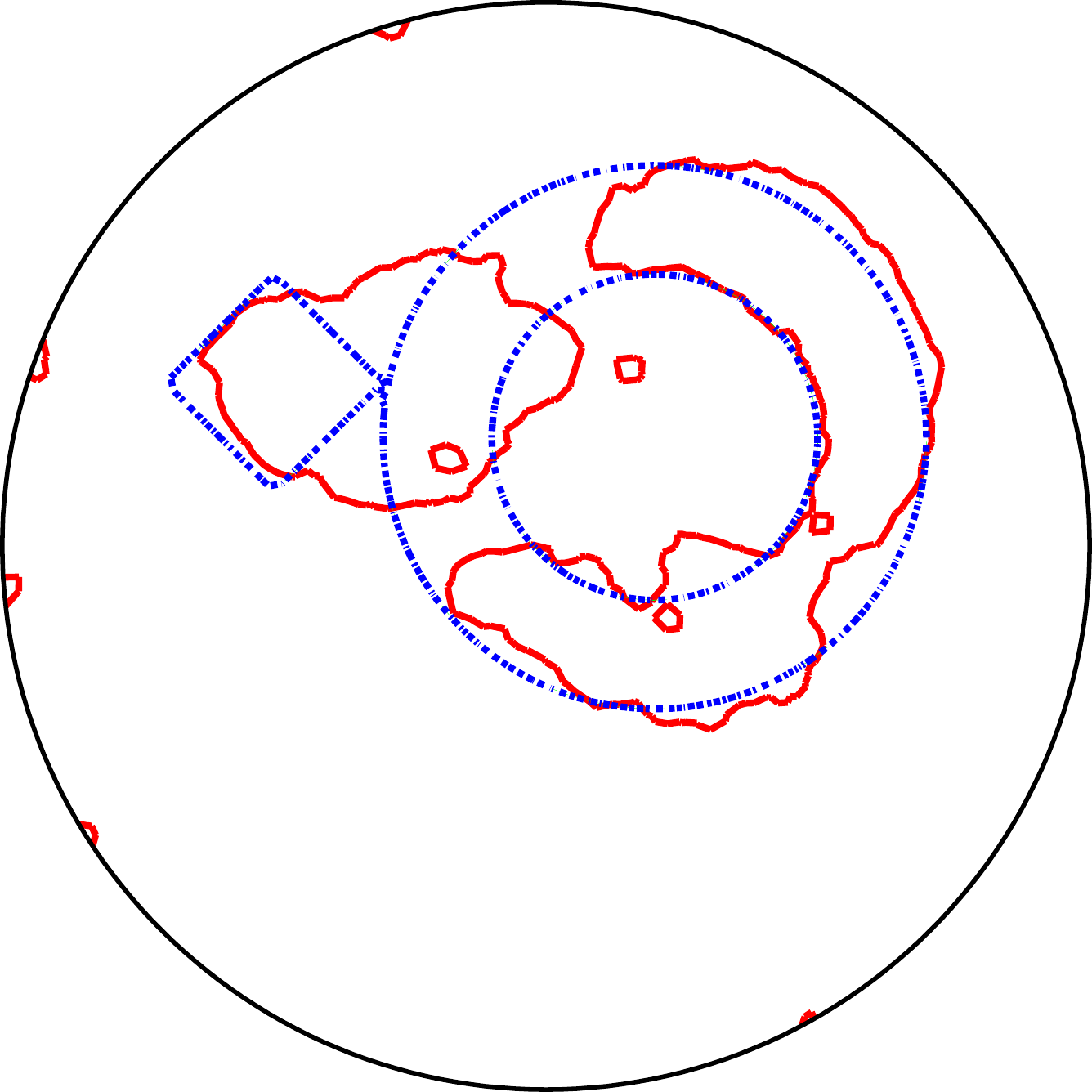}}
  \caption{TSCM: experiment 2; $\rho = 1\%$; $N(\omega) =%
4${\label{fig:multifreqTorus}}}
\end{center}
\end{figure}

Last, we examine the behavior of the TSCM regarding $\Delta \lambda$, i.e.
regarding the number of $\lambda-$iterations $N(\lambda)$. We take the noise
level of $1\%$ and $N(\omega)=2.$ In Figure \ref{fig:dependencyonlambda}(a) the
total number of iterations $n$ of Algorithm \ref{alg:1} and in Figure
\ref{fig:dependencyonlambda}(b) the corresponding relative error of the
conductivity $e(\sigma)$ are plotted against $ln(N(\lambda)).$ We see that $n$
shows tendency to grow and $e(\sigma)$ tendency to decrease. The results are
obtained from a single-problem sample for each $N(\lambda)$. In Figure
\ref{fig:dependencyonlambdaexamples} two particular examples are presented for
$N(\lambda)=2$ and for $N(\lambda)=4.$ We see that to correctly identify the
shape and particularly its topology, it is necessary to consider at least
$N(\lambda)=4$. The continuation method has to be allowed to perform a
sufficient number of steps to shift the information from $\sldva$ to $\spc$,
i.e. the process has to be sufficiently continuous.

\begin{figure}[ht]
\begin{center}
  \subfigure[Number of
iterations]{\includegraphics[width=0.45\columnwidth]{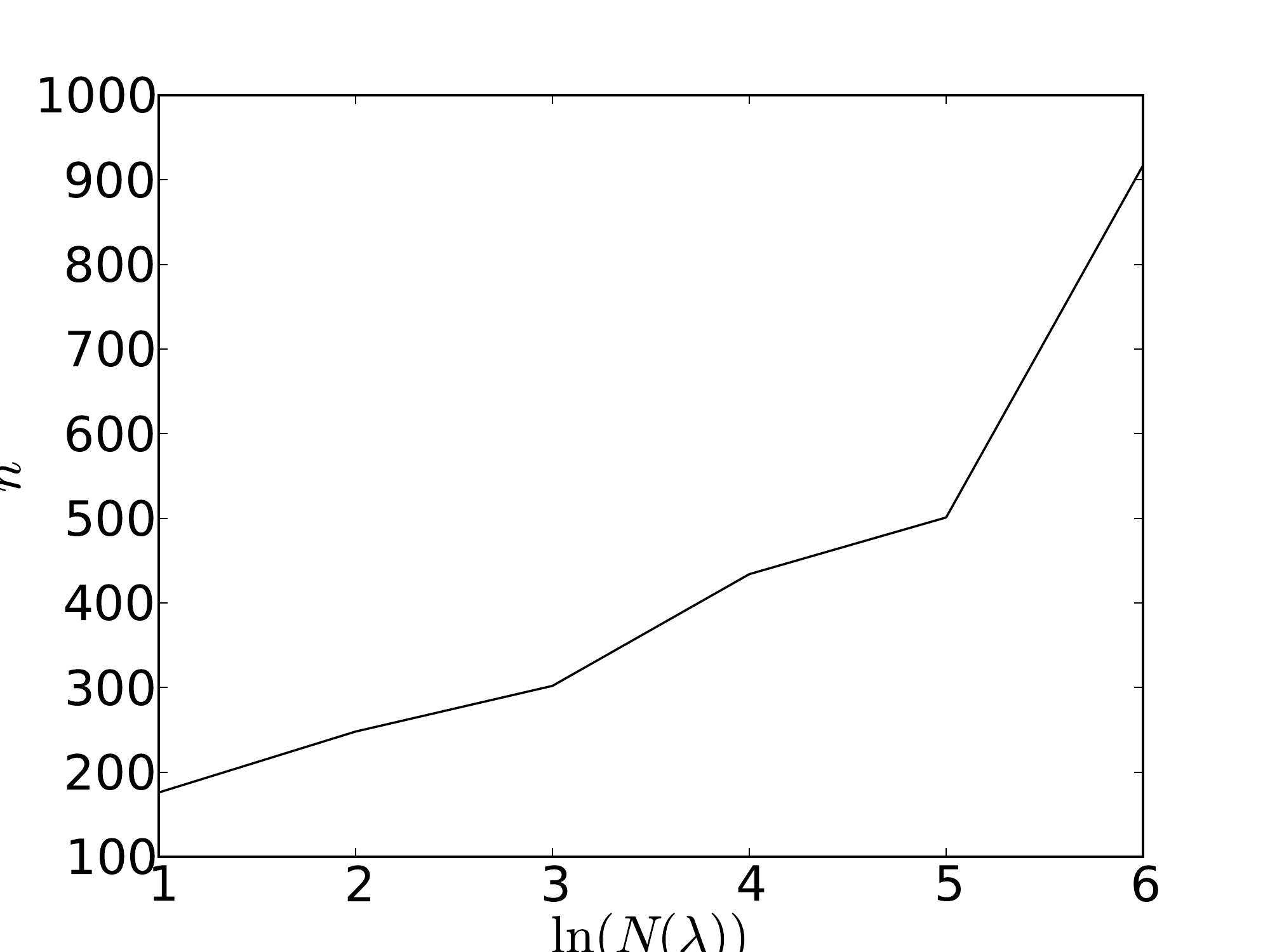}}
  \subfigure[Relative error of
sigma]{\includegraphics[width=0.45\columnwidth]{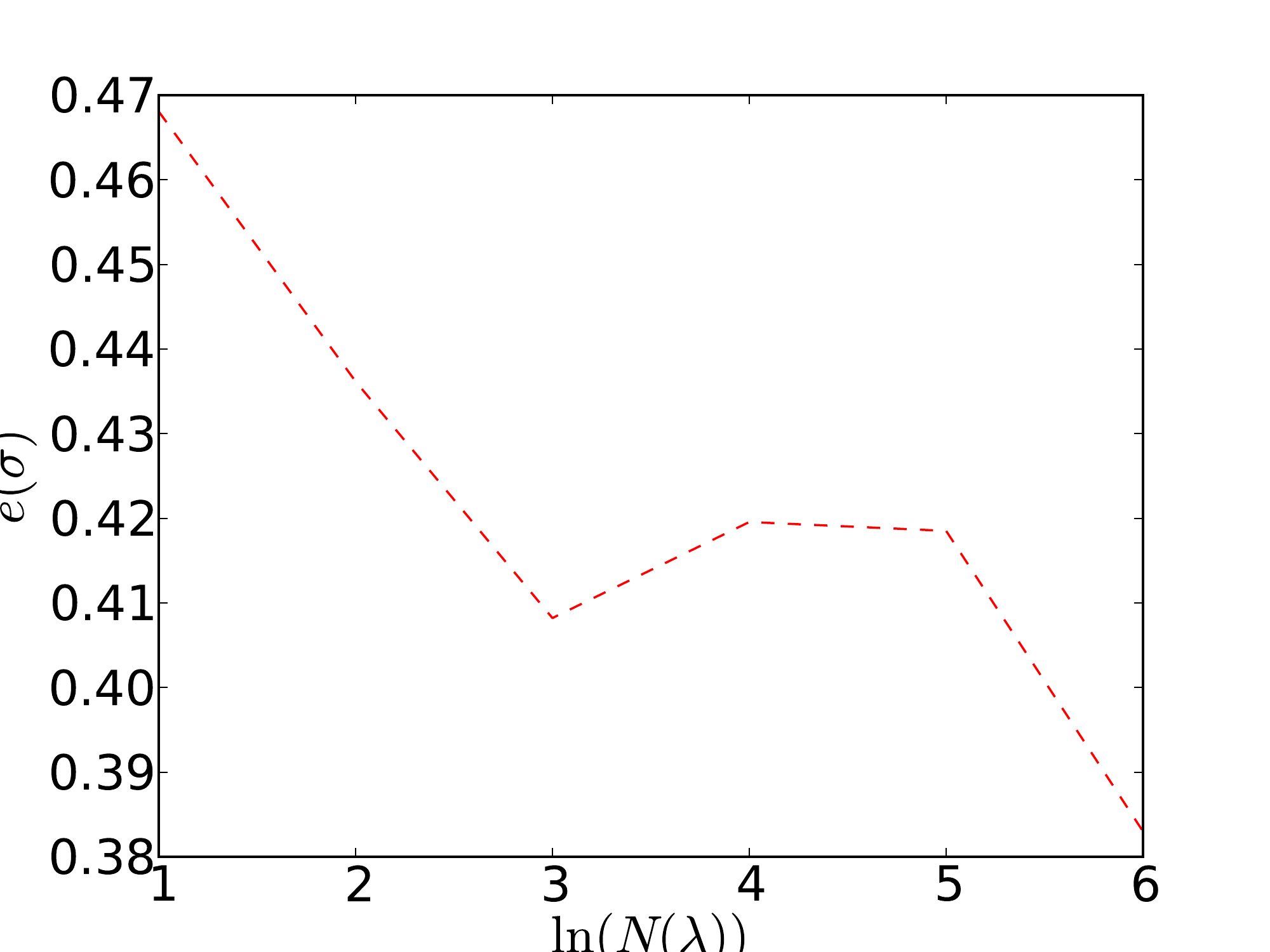}}
  \caption{TSCM: dependency on $\Delta \lambda$; $\rho = 1\%$; $N(\omega) =
2${\label{fig:dependencyonlambda}}}
\end{center}
\end{figure}

\begin{figure}[ht]
\begin{center}
  \subfigure[$N(\lambda)=2$, final
$\phi_{175}$]{\includegraphics[width=0.35\columnwidth]{%
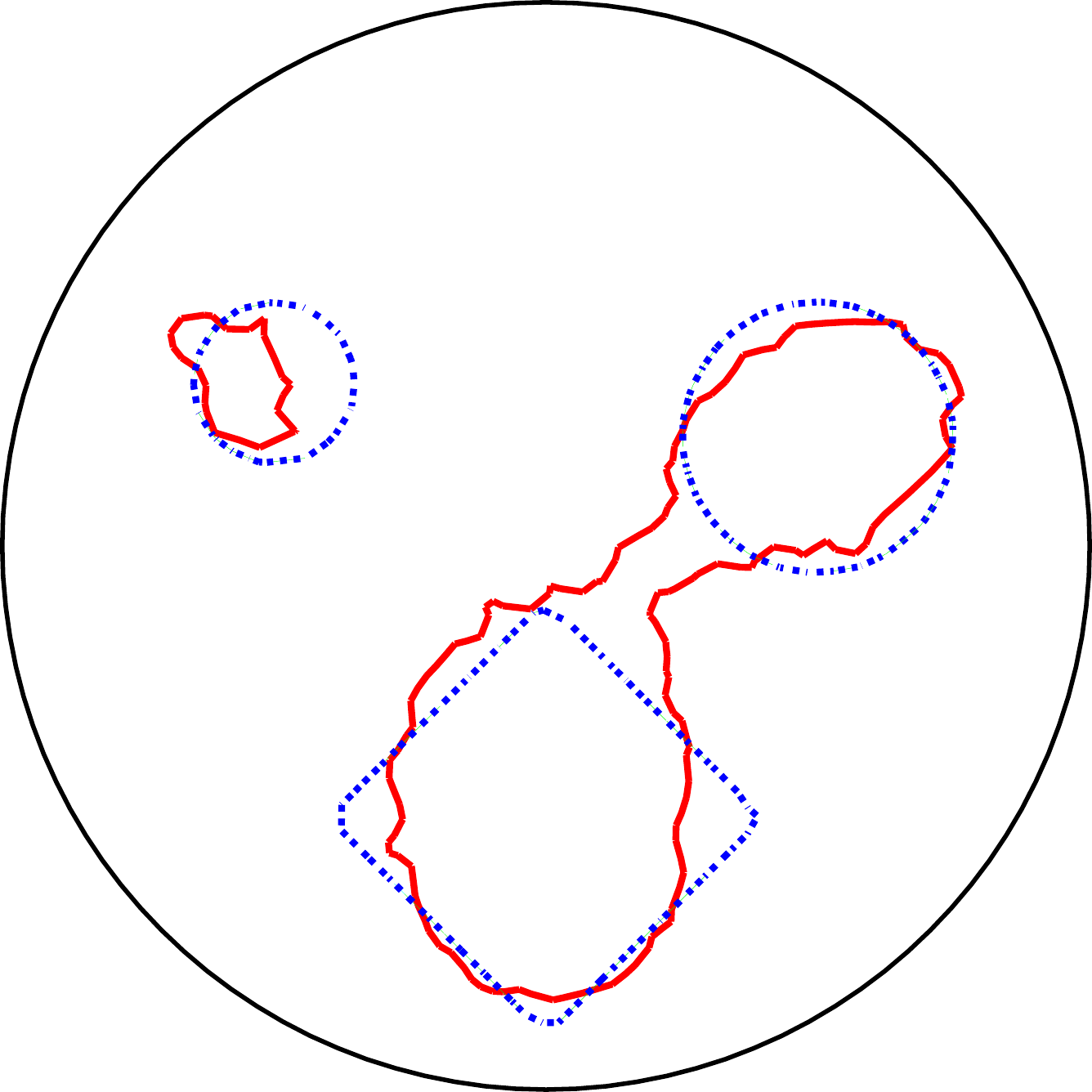}}
\hspace{1cm}
  \subfigure[$N(\lambda)=4$, final
$\phi_{247}$]{\includegraphics[width=0.35\columnwidth]{%
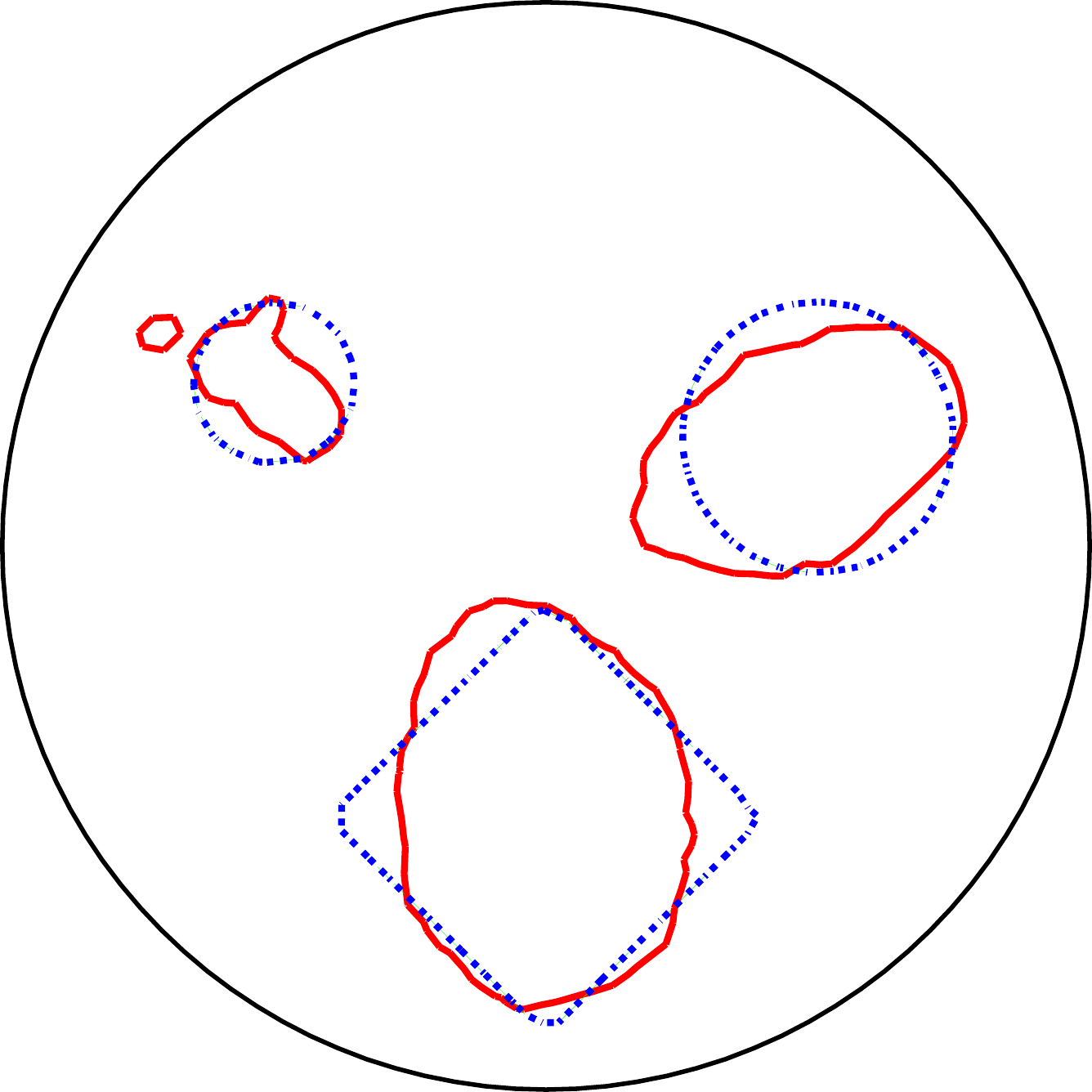}}
  \caption{TSCM: dependency on $\Delta \lambda$; $\rho = 1\%$; $N(\omega) =
2${\label{fig:dependencyonlambdaexamples}}}
\end{center}
\end{figure}

\section{Conclusions}
In this paper we have presented a continuation approach for Tikhonov
regularization and employed it to perform shape identification without any
initial knowledge of topology. We have successfully applied the resulting
topology-to-shape continuation method (TSCM) to a magnetic
induction tomography (MIT) problem.

This method appears to be a very promising candidate for an ultimate framework 
unifying both topology and shape sensitivities. To establish such a claim
more rigorously, it is necessary to provide a deeper analysis of the
continuation approach with respect to the homotopy parameter $\lambda$, which is
a possible future work. Any result in this direction will be dependent on a
particular choice of the functional spaces $W$ and $U$ and their properties.
Our understanding of the underlying concepts suggests that for
TSCM-specific choice of the functional spaces such an analysis is attainable.

In this paper we have provided more or less standard results on well-posedness,
stability ad convergence of the framework. Under a strong condition of
uniqueness, we have provided a local correctness result of the continuation
approach (Theorem \ref{thm:main}).

The numerical results of the TSCM for multiple-frequency MIT show decent
accuracy and above all excellent stability of the reconstruction with respect to
noise. It suggest that generalization to multiple-valued piecewise-constant
parameters scenario is reasonable and should be fairly straightforward. As
already known for MIT, simultaneous reconstruction of both conductivity and
permittivity is possible. Altogether, the MIT with the TSCM as a solver could be
used as a diagnostic method.

Possible future work with respect to the TSCM or to the continuation
approach in general is to propose and analyze appropriate parameter choice rules
(PCRs) for the two regularization parameters $\alpha$ and $\beta$ in
\refe{eq:tikhonov_functional_with_continuation}. The regularization parameters
could be considered as functions of $\lambda$ as well. This should lead to
$\lambda$-adaptive PCRs and consequently a more efficient implementation of the
TSCM algorithm. From the numerical point of view also conjugate gradient,
quasi-Newton or Gauss–Newton algorithm extensions are possible.

\section*{Acknowledgement}
Valdemar Melicher would like to acknowledge the support of the BOF
doctor-assistant research mandate 01P09209T of Ghent University,
Ghent, Belgium. Vladim\'ir Vr\'abe\ml\  was supported by the BOF-grant number
01D00409 of Ghent
University.

\bibliographystyle{unsrt}
\bibliography{ref}

\begin{thebibliography}{10}

\bibitem{Hofmann2007}
B.~Hofmann, B.~Kaltenbacher, C.~P{\"o}schl, and O.~Scherzer.
\newblock A convergence rates result for {T}ikhonov regularization in {B}anach
  spaces with non-smooth operators.
\newblock {\em Inverse Problems}, 23(3):987--1010, 2007.

\bibitem{Engl1996}
H.~W. Engl, M.~Hanke, and A.~Neubauer.
\newblock {\em Regularization of Inverse Problems}, volume 375 of {\em
  Mathematics and its Applications}.
\newblock Kluwer Academic Publishers, Dordrecht, 1996.

\bibitem{Morozov1984}
V.~A. Morozov.
\newblock {\em Methods for Solving Incorrectly Posed Problems}.
\newblock Berlin: Springer, 1984.

\bibitem{Allgower2003}
Eugene~L. Allgower and Kurt Georg.
\newblock {\em Introduction to Numerical Continuation Methods}.
\newblock SIAM, 2003.

\bibitem{Zeidler1989}
E.~Zeidler and L.F. Boron.
\newblock {\em {Nonlinear Functional Analysis and Its Applications: Part 2 B:
  Nonlinear Monotone Operators}}.
\newblock Nonlinear Functional Analysis and Its Applications. Springer, 1989.

\bibitem{Zeidler1985}
E.~Zeidler.
\newblock {\em {Nonlinear Functional Analysis and Its Applications: Fixed point
  theorems}}.
\newblock Nonlinear Functional Analysis and Its Applications. Springer-Verlag,
  1985.

\bibitem{Borcea2002}
L.~Borcea.
\newblock Electrical impedance tomography.
\newblock {\em Inverse Problems}, 18(6):R99, 2002.

\bibitem{Acar1994}
R.~Acar and C.~R. Vogel.
\newblock Analysis of bounded variation penalty methods for ill-posed problems.
\newblock {\em Inverse Problems}, 10(6):1217--1229, 1994.

\bibitem{Pironneau1984}
O.~Pironneau.
\newblock {\em Optimal Shape Design for Elliptic Systems}.
\newblock Springer Netherlands (Springer-Verlag, Berlin, Heidelberg, New York,
  Tokyo), 1984.
\newblock ISBN: 0-3871-2069-6.

\bibitem{Sokolowski1992}
J.~Soko{\l}owski and J.~P. Zol\'esio.
\newblock {\em Introduction to shape optimization. Shape sensitivity
  analysis.}, volume~16.
\newblock Springer-Verlag, 1992.

\bibitem{Santosa1996}
F.~Santosa.
\newblock A level-set approach for inverse problems involving obstacles.
\newblock {\em ESAIM Contr\^ole Optim. Calc. Var.}, 1:17--33, 1996.

\bibitem{Fang2003}
W.~Fang and K.~Ito.
\newblock Identification of contact regions in semiconductor transistors by
  level-set methods.
\newblock {\em {Journal of Computational and Applied Mathematics}},
  159(2):399--410, 2003.

\bibitem{Chan2005}
T.~Chan, S.~Esedoglu, F.~Park, and A.~Yip.
\newblock {\em The Handbook of Mathematical Models in Computer Vision}, chapter
  Total Variation Image Restoration: Overview and Recent Developments, pages
  17--32.
\newblock Springer, 2005.

\bibitem{Bendsoe1988}
M.P. Bends{\o}e and N.~Kikuchi.
\newblock Generating optimal topologies in structural design using a
  homogenization method.
\newblock {\em Computer Methods in Applied Mechanics and Engineering},
  71(2):197--224, 1988.

\bibitem{Allaire2002}
G.~Allaire.
\newblock {\em {Shape Optimization by the Homogenization Method}}, volume 146
  of {\em Applied Mathematical Sciences}.
\newblock Springer, 2002.

\bibitem{Eschenauer1994}
H.~A. Eschenauer, V.~V. Kobelev, and A.~Schumacher.
\newblock Bubble method for topology and shape optimization of structures.
\newblock {\em Structural and Multidisciplinary Optimization}, 8:42--51, 1994.
\newblock 10.1007/BF01742933.

\bibitem{Schumacher1996}
A.~Schumacher.
\newblock {\em Topologieoptimierung von Bauteilstrukturen unter Verwendung von
  Lochpositionierungkriterien}.
\newblock PhD thesis, Siegen University, Siegen, Germany, 1996.

\bibitem{Sokolowski1997}
J.~Soko{\l}owski and A.~Zochowski.
\newblock On topological derivative in shape optimisation.
\newblock Technical Report 3170, INRIA-Lorraine, 1997.

\bibitem{Sokolowski1999}
J.~Soko{\l}owski and A.~Zochowski.
\newblock On the topological derivative in shape optimization.
\newblock {\em SIAM Journal on Control and Optimization}, 37(4):1251--1272,
  1999.

\bibitem{Sokolowski1999a}
J.~Soko{\l}owski and A.~Zochowski.
\newblock Topological derivatives for elliptic problems.
\newblock {\em Inverse Problems}, 15(1):123--134, 1999.

\bibitem{Burger2004a}
M.~Burger, B.~Hackl, and W.~Ring.
\newblock Incorporating topological derivatives into level set methods.
\newblock {\em Journal of Computational Physics}, 194(1):344--362, 2004.

\bibitem{He2007}
Lin He, Chiu-Yen Kao, and Stanley Osher.
\newblock {Incorporating topological derivatives into shape derivatives based
  level set methods}.
\newblock {\em Journal of Computational Physics}, {225}({1}):{891--909}, {JUL
  1} {2007}.

\bibitem{Allaire2004}
G.~Allaire, F.~Jouve, and A.-M. Toader.
\newblock Structural optimization using sensitivity analysis and a level-set
  method.
\newblock {\em Journal of Computational Physics}, 194(1):363--393, 2004.

\bibitem{Allaire2005}
G.~Allaire, F.~de~Gournay, F.~Jouve, and A.-M. Toader.
\newblock {Structural optimization using topological and shape sensitivity via
  a level set method}.
\newblock {\em Control and Cybernetics}, {34}({1}):{59--80}, {2005}.

\bibitem{Nielsen2007}
L.~K. Nielsen, X.-C. Tai, Si.~I. Aanonsen, and M.~Espedal.
\newblock {A binary level set model for elliptic inverse problems with
  discontinuous coefficients}.
\newblock {\em {INTERNATIONAL JOURNAL OF NUMERICAL ANALYSIS AND MODELING}},
  {4}({1}):{74--99}, {2007}.

\bibitem{Zhu2011}
S.~Zhu, Q.~Wu, and C.~Liu.
\newblock Shape and topology optimization for elliptic boundary value problems
  using a piecewise constant level set method.
\newblock {\em Applied Numerical Mathematics}, 61(6):752--767, 2011.

\bibitem{Cezaro2013}
{De Cezaro, A. and Leitao, A. and Tai, X.-C.}
\newblock {On piecewise constant level-set (PCLS) methods for the
  identification of discontinuous parameters in ill-posed problems}.
\newblock {\em {Inverse Problems}}, {29}:{015003 (23 pp.)}, {Jan.} {2013}.

\bibitem{Hintermuller2008}
M.~Hintermueller and A.~Laurain.
\newblock {Electrical impedance tomography: from topology to shape}.
\newblock {\em {Control and Cybernetics}}, {37}({4, SI}):{913--933}, {2008}.

\bibitem{Fulmanski2008}
P.~Fulmanski, A.~Laurain, J.-F. Scheid, and J.~Sokolowski.
\newblock Level set method with topological derivatives in shape optimization.
\newblock {\em Int. J. Comput. Math.}, 85(10):1491--1514, October 2008.

\bibitem{Tai2007}
X.-C. Tai and H.~Li.
\newblock {A piecewise constant level set method for elliptic inverse
  problems}.
\newblock {\em {APPLIED NUMERICAL MATHEMATICS}}, {57}({5-7}):{686--696},
  {MAY-JUL} {2007}.
\newblock {International Conference on Scientific Computing (ICSC05), Nanjing
  Univ, Nanjing, PEOPLES R CHINA, JUN 04-08, 2005}.

\bibitem{Griffiths2001}
H.~Griffiths.
\newblock Magnetic induction tomography.
\newblock {\em Measurement Science and Technology}, 12:1126--1131, 2001.

\bibitem{Soleimani2008}
M.~Soleimani.
\newblock Computational aspects of low frequency electrical and electromagnetic
  tomography: a review study.
\newblock {\em International Journal For Numerical Analysis and Modeling},
  5(3):407--440, 2008.

\bibitem{Cheney1999}
M.~Cheney, D.~Isaacson, and J.C. Newell.
\newblock {Electrical impedance tomography}.
\newblock {\em {SIAM Review}}, {41}({1}):{85--101}, {MAR} {1999}.

\bibitem{Brunner2006}
P.~Brunner, R.~Merwal, A.~Missner, J.~Rosell, K.~Hollaus, and H.~Scharfetter.
\newblock Reconstruction of the shape of conductivity spectra using
  differential multi-frequency magnetic induction tomography.
\newblock {\em Physiological Measurement}, 27:237--248, 2006.

\bibitem{Zolgharni2009}
M.~Zolgharni, P.~D. Ledger, and Griffiths H.
\newblock Forward modelling of magnetic induction tomography: a sensitivity
  study for detecting haemorrhagic cerebral stroke.
\newblock {\em Medical and Biological Engineering and Computing},
  47(12):1301--1313, 2009.

\bibitem{Korjenevsky2000}
A.~Korjenevsky, V.~Cherepin, and S.~Sapetsky.
\newblock Magnetic induction tomography: experimental realization.
\newblock {\em Physiological Measurement}, 21:89--94, 2000.

\bibitem{Kufner1977}
Alois Kufner, Oldrich John, and Svatopluk Fu\v{c}\'\i{}k.
\newblock {\em Function Spaces}.
\newblock Monographs and Textbooks on Mechanics of Solids and Fluids;
  Mechanics: Analysis. Noordhoff International Publishing, Leyden; Academia,
  Prague, 1977.

\bibitem{Osher1988}
S.~Osher and J.A. Sethian.
\newblock Fronts propagating with curvature dependent speed: algorithms based
  on {Hamilton-Jacobi} formulations.
\newblock {\em J. Comput. Phys.}, 79:12--49, 1988.

\bibitem{Cimrak2010}
I.~Cimrák and V.~Melicher.
\newblock Determination of precession and dissipation parameters in
  micromagnetism.
\newblock {\em Journal of Computational and Applied Mathematics},
  234(7):2239--2249, 2010.

\bibitem{Hecht2009}
Fr\'{e}d\'{e}ric Hecht, Olivier Pironneau, Jacques Morice, Antoine Le~Hyaric,
  and Kohji Ohtsuka.
\newblock {\em FreeFem++}.
\newblock Laboratoire Jacques-Louis Lions, Universit\'{e} Pierre et Marie
  Curie, Paris, 3rd edition, May 2009.
\newblock http://www.freefem.org/ff++, {V}ersion 3.2.

\bibitem{Chung2005}
ET~Chung, TF~Chan, and XC~Tai.
\newblock {Electrical impedance tomography using level set representation and
  total variational regularization}.
\newblock {\em {JOURNAL OF COMPUTATIONAL PHYSICS}}, {205}({1}):{357--372}, {MAY
  1} {2005}.

\end{thebibliography}
\end{document}